\documentclass{amsart}
\usepackage{fullpage}
\usepackage{amsmath, amssymb}
\usepackage{amsthm}
\usepackage{amscd}
\usepackage{color}
\usepackage{tikz}
\usepackage[all,cmtip]{xy}
\usepackage{xcolor}
\usepackage{pdfcomment}
\usepackage{pbox}
\usepackage{mathrsfs}
\usetikzlibrary{snakes}
\hypersetup{hidelinks}

\newtheorem{thm}{Theorem}[section]
\newtheorem{proposition}[thm]{Proposition}
\newtheorem{corollary}[thm]{Corollary}
\newtheorem{lemma}[thm]{Lemma}
\newtheorem{definition}[thm]{Definition}

\newtheorem{remark}[thm]{Remark}

\newtheorem{notation}{Notation}

\newcommand{\Hom}{\mathrm{Hom}}
\newcommand{\add}{\mathrm{add}}

\newcommand{\cX}{\mathcal{X}}

\newcommand{\cC}{\mathcal{C}}

\newcommand{\cP}{\mathcal{P}}
\newcommand{\cT}{\mathcal{T}}

\newcommand{\cZ}{\mathcal{Z}}
\newcommand{\cY}{\mathcal{Y}}

\newcommand{\bZ}{\mathbb{Z}}

\newcommand{\bK}{\mathbb{K}}

\begin{document}

\title{Lattices of t-structures and thick subcategories for discrete cluster categories}
\author{Sira Gratz and Alexandra Zvonareva}
\thanks{The first author acknowledges support through VILLUM FONDEN project number 42076, and EPSRC grant EP/V038672/1.}
\address{Sira Gratz, Department of Mathematics, Aarhus University, Ny Munkegade 118, 8000 Aarhus C, Denmark}
\email{sira@math.au.dk}
\address{Alexandra Zvonareva, Institut f\"ur Algebra und Zahlentheorie, Universit\"at Stuttgart, Pfaffenwaldring 57, 70569 Stuttgart, Germany}
\email{alexandra.zvonareva@mathematik.uni-stuttgart.de}
\maketitle
\vspace{-30pt}
\begin{abstract}
\noindent We classify t-structures and thick subcategories in any discrete cluster category $\cC(\cZ)$ of Dynkin type $A$, and show that the set of all t-structures on $\cC(\cZ)$ is a lattice under inclusion of aisles, with meet given by their intersection. We show that both the lattice of t-structures on $\cC(\cZ)$ obtained in this way and the lattice of thick subcategories of $\cC(\cZ)$ are intimately related to the lattice of non-crossing partitions of type $A$. In particular, the lattice of equivalence classes of non-degenerate t-structures on such a category is isomorphic to the lattice of non-crossing partitions of a finite linearly ordered set.
\end{abstract}

\tableofcontents
\vspace{-20pt}
\section{Introduction}
The notion of a t-structure, first introduced by Beilinson, Bernstein, Deligne and Gabber (see \cite{BBD} for the latest edition), lies behind several groundbreaking developments in modern mathematics. A t-structure on a triangulated category  simultaneously provides a decomposition of the category and a generalized notion of homology taking values in an embedded abelian category. This concept is crucial in defining perverse sheaves and stability conditions and has deep relations to tilting theory, silting theory and uniqueness of enhancements.

Given a triangulated category, there is a natural partial order on its collection of t-structures given by inclusion of aisles. This is the order introduced by Aihara and Iyama \cite{AI} for the particular case of t-structures generated by silting objects and studied in silting theory. This poset is poorly understood in general. In particular, one would like to know when meets and joins exist, or, even more ambitiously, when the intersection of two or more aisles, or the extension closure of their union, is again an aisle.
This problem was addressed by Bondal \cite{Bondal} for the intersection of aisles and coaisles and by Broomhead, Pauksztello and Ploog \cite{BBP} for the intersection of coaisles. Both provide some sufficient condition for the intersection of two (co)aisles to be a (co)aisle, the former in terms of {\em consistent pairs}, and the latter in terms of a {\em refined truncation algorithm}. The question of when these conditions hold remains open.
In particular, \cite{Bondal} shows that the lattice structure obtained from consistent pairs of t-structures satisfies strong modularity and distributivity conditions that severely restrict the kind of lattices we can obtain in this way.

In this paper we present a class of categories whose t-structures form a lattice under inclusion of aisles, with meet given by intersection: the discrete cluster categories of Dynkin type $A$, as introduced by Igusa and Todorov in \cite{IT}. The lattices we obtain in this way are intimately related to non-crossing partitions of type $A$. In particular, they are, in general, neither distributive nor modular, and thus do not fall under the framework set up in \cite{Bondal}, providing a fresh perspective on the question of when we can expect t-structures to form a lattice under the expected operations.

A discrete cluster category of Dynkin type $A$ is a triangulated category $\cC(\cZ)$ associated to a discrete subset $\cZ \subseteq S^1$ of the unit circle with $n < \infty$ limit points. It demonstrates cluster combinatorics of Dynkin type $A$. In particular, it has cluster tilting subcategories, classified by Gratz, Holm and J{\o}rgensen in \cite{GHJ}, given by suitable triangulations of the closed disc with marked points $\cZ$. While providing a familiar environment reminiscent of finite type $A$ cluster combinatorics, these categories exhibit phenomena that remain hidden in the finite rank setting. Thus cluster tilting subcategories have infinitely many isoclassses of indecomposable objects, as opposed to the finite rank case, the suspension functor is no longer periodic, and we now find non-trivial t-structures. For the case $n=1$, these were first classified by Ng in \cite{Ng}. We generalise this result for $n \geq 2$, and connect it to non-crossing partitions of $[n] = \{1, \ldots, n\}$. More specifically, t-structures in $\cC(\cZ)$ are classified by {\em $\overline{\cZ}$-decorated non-crossing partitions}, that is, by pairs $(\cP,{\bf x})$, where $\cP$ is a non-crossing partition of $[n]$ and ${\bf x}$ is a compatible point in $\overline{\cZ}^n$, where $\overline{\cZ}$ denotes the topological closure of $\cZ$ (cf.\ Definition \ref{D:Zdecoratedncpartition} for the precise definition).

\begin{thm}[Theorem \ref{T:one-to-one}]\label{thm A}
There is a one-to-one correspondence
	\[
		\{ \overline{\cZ}\text{-decorated non-crossing partitions of $[n]$}\} \to \{\text{t-structures of $\cC(\cZ)$}\}.
	\]
	It sends a $\overline{\cZ}$-decorated non-crossing partition $(\cP, (x_1, \ldots, x_n))$ to the t-structure with aisle 
	\[
	\cX(\cP,{\bf x}) = \add \{\{y_0,y_1\} \; \text{arc of}\; \cZ \mid y_0,y_1 \in \bigcup_{i \in B} (a_i,x_i] \; \text{for some} \; B \in \cP\}.
	\]
\end{thm}

A t-structure is uniquely determined by its aisle, and the correspondence from Theorem \ref{thm A} explicitly connects a $\overline{\cZ}$-decorated non-crossing partition with an aisle represented in terms of the combinatorial model of the disc with marked points $\cZ$. The duality of the aisle-coaisle relationship is reflected in the combinatorial model: The coaisle of a t-structure is naturally associated to the Kreweras complement \cite{K} of the non-crossing partition corresponding to its aisle. Using this description, we also describe the approximation triangles for any t-structure on $\cC(\cZ)$ combinatorially. Having classified all t-structures of $\cC(\cZ)$, we can restrict to those with specific properties. While there are no bounded t-structures, we have bounded above t-structures, described by the coarsest non-crossing partition $\{[n]\}$ of $[n]$, and bounded below t-structures, described by the finest non-crossing partition $\{\{1\},\ldots,\{n\}\}$ of $[n]$. Furthermore, the non-degenerate t-structures of $\cC(\cZ)$ correspond precisely to those $\overline{\cZ}$-decorated non-crossing partitions whose decorations lie in $\cZ^n$.
The classification of t-structures on $\cC(\cZ)$ allows us to prove that they form a lattice under inclusion of aisles, as promised above.

\begin{thm}[Theorem \ref{T:lattice}, Proposition \ref{P:meet is intersection}]
	The set of t-structures of $\cC(\cZ)$ forms a lattice under inclusion of aisles. The meet is given by the intersection of aisles.
\end{thm}

Another prominent class of subcategories of triangulated categories are thick subcategories. They have been classified for a number of important classes of triangulated categories, e.g., in classical work by Devinatz, Hopkins and Smith \cite{DHS} and Neeman \cite{Neemanthick} for the perfect derived category of a commutative noetherian ring and by Benson, Carlson and Rickard \cite{BCR} for the  stable category of the group algebra of a finite group. 
Recent advances in the classification of thick subcategories have been made in several areas, for instance in tensor triangular geometry starting with Balmer's work \cite{Bal}, algebraic geometry such as in work by Takahashi \cite{Takahashi} and Elagin and Lunts \cite{EL} and representation theory such as in work by Broomhead \cite{Broomhead}. For a list of classifications of thick tensor ideals, see the list in \cite{Balguide}.   

Contrary to aisles of t-structures thick subcategories always form a lattice under inclusion. Specific connections between lattices of thick subcategories and lattices of non-crossing partitions have been established in the examples studied by Ingalls and Thomas \cite{IngTh} and Gratz and Stevenson \cite{GS}, and explored more generally by Krause in \cite{Krause:ncthick}.
Thick subcategories in $\cC(\cZ)$ are given by {\em non-exhaustive non-crossing partitions of $[n]$}, i.e.\ by non-crossing partitions of subsets of $[n]$. We denote the lattice of non-exhaustive non-crossing partitions by $NNC_n$.

\begin{thm}[Theorem \ref{T:thick}]
There is an isomorphism of lattices 
\[
		NNC_n \cong \mathrm{Thick}(\cC(\cZ)).
	\]
\end{thm}

Having made implicit use of the beauty of non-crossing partitions to better our understanding of various structures on $\cC(\cZ)$, it is only fair that they should make an unadulterated appearance. Indeed, they describe the  equivalence classes (in the sense of Neeman \cite{NeEqT}) of non-degenerate t-structures, see Corollary \ref{C:lattice of equivalence classes}.

\section{Set-up}

In this paper we will consider $\bK$-linear triangulated categories over some field $\bK$. The shift functor in a triangulated category will be usually denoted by $\Sigma$, we will also sometimes call it suspension. All subcategories are assumed to be full and closed under direct sums and summands. 

In order to discuss discrete cluster categories of type $A$, we first introduce the appropriate combinatorial model via arcs in a closed disc.
Let us denote by $S^1$ the unit circle, considered with the anti-clockwise order. Namely, 
 for pairwise distinct $x,y,z \in S^1$ we write
\[
	x < y < z 
\]
if and only if when we go in an anti-clockwise direction, we encounter first $x$, then $y$, then $z$.
For $m \geq 3$ and $x_1, \ldots, x_m \in S^1$ we write
\[
	x_1 < x_2 < \ldots < x_m
\] 
if and only if $x_i < x_{i+1} < x_{i+2}$ for all $1 \leq i \leq m$, where for this purpose we consider the indices modulo $m$. In particular, this also implies that $x_{m-1} < x_m < x_1$ and $x_m<x_1<x_2$.

Let us fix a  discrete, infinite set $\cZ \subseteq S^1$. Set $L(\cZ) = \overline{\cZ} \setminus \cZ$ to be the set of limit points, where $\overline{\cZ}$ denotes the topological closure of $\cZ$, and assume that $|L(\cZ)| = n \in \bZ_{>0}$, and that every limit point is a two-sided limit point.  Note that usually such a subset is called {\bf admissible}.
We label the limit points of $\cZ$ by $a_1, \ldots, a_n$, where 
\[
	a_1 < a_2 < \ldots < a_{n-1} < a_n,
\]
see for example Figure \ref{fig:admissible subset}. For ease of notation, throughout we will consider the indices modulo $n$.
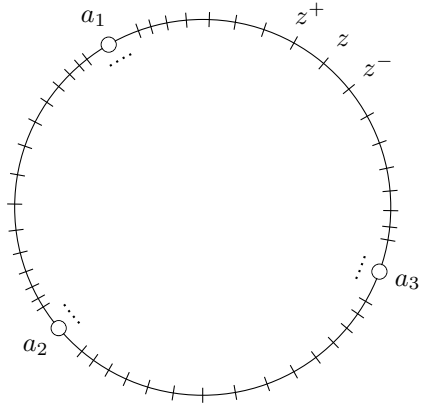
\begin{figure}[h]
\begin{center}
\begin{tikzpicture}[scale=5]
        \draw (0,0) circle(0.5cm);
	
	\node at (50:0.58){$z$};
	\node at (39:0.60){$z^-$};
	\node at (61:0.59){$z^+$};

    \draw  (128:0.48) -- (128:0.52);
    \draw  (131:0.48) -- (131:0.52);
    \draw  (135:0.48) -- (135:0.52);
    \draw  (140:0.48) -- (140:0.52);
    \draw  (146:0.48) -- (146:0.52);
    \draw  (153:0.48) -- (153:0.52);
    \draw  (161:0.48) -- (161:0.52);
	\draw  (170:0.48) -- (170:0.52);
    \draw  (179:0.48) -- (179:0.52);
    \draw  (187:0.48) -- (187:0.52);
    \draw  (194:0.48) -- (194:0.52);
    \draw  (200:0.48) -- (200:0.52);
    \draw  (205:0.48) -- (205:0.52);
    \draw  (209:0.48) -- (209:0.52);
    \draw  (212:0.48) -- (212:0.52);
    
    \draw  (230:0.48) -- (230:0.52);
    \draw  (234.5:0.48) -- (234.5:0.52);
    \draw  (240:0.48) -- (240:0.52);
    \draw  (246:0.48) -- (246:0.52);
    \draw  (253:0.48) -- (253:0.52);
    \draw  (261:0.48) -- (261:0.52);
    \draw  (270:0.48) -- (270:0.52);    
	\draw  (280:0.48) -- (280:0.52);
    \draw  (290:0.48) -- (290:0.52);
    \draw  (299:0.48) -- (299:0.52);
    \draw  (307:0.48) -- (307:0.52);
    \draw  (314:0.48) -- (314:0.52);
    \draw  (320:0.48) -- (320:0.52);
    \draw  (325.5:0.48) -- (325.5:0.52);
    \draw  (330:0.48) -- (330:0.52);

    \draw  (-10:0.48) -- (-10:0.52);
    \draw  (-6:0.48) -- (-6:0.52);
    \draw  (-1:0.48) -- (-1:0.52);
    \draw  (5:0.48) -- (5:0.52);
    \draw  (12:0.48) -- (12:0.52);
    \draw  (20:0.48) -- (20:0.52);
    \draw  (29:0.48) -- (29:0.52);
    \draw  (39:0.48) -- (39:0.52);
	\draw  (50:0.48) -- (50:0.52);
    \draw  (61:0.48) -- (61:0.52);
    \draw  (71:0.48) -- (71:0.52);
    \draw  (80:0.48) -- (80:0.52);
    \draw  (88:0.48) -- (88:0.52);
    \draw  (95:0.48) -- (95:0.52);
    \draw  (101:0.48) -- (101:0.52);
    \draw  (106:0.48) -- (106:0.52);
    \draw  (110:0.48) -- (110:0.52);

	\draw (120:0.5) node[fill=white,circle,inner sep=0.065cm] {} circle (0.02cm);
	\node at (120:0.58){$a_1$};
	\draw (220:0.5) node[fill=white,circle,inner sep=0.065cm] {} circle (0.02cm);
	\node at (220:0.58){$a_2$};
	\draw (340:0.5) node[fill=white,circle,inner sep=0.065cm] {} circle (0.02cm);
	\node at (340:0.58){$a_3$};
	\draw[thick,dotted] (115:0.45) arc (115:125:0.45);
	\draw[thick,dotted] (215:0.45) arc (215:225:0.45);
	\draw[thick,dotted] (335:0.45) arc (335:345:0.45);		
	
  \end{tikzpicture}
  \end{center}
  \caption{Admissible subset $\cZ\subseteq S^1$ with $|L(\cZ)| = 3$. The limit points are marked with small circles.  
\label{fig:admissible subset}}
\end{figure}

For $a,b \in \overline{\cZ}$, we denote the interval between $a$ and $b$ when going in an anti-clockwise direction by
\[
	(a,b) = \{c \in \overline{\cZ} \mid a < c < b \}.
\]
Note that this interval may contain limit points. Analogously we define the closed, and half-open intervals $[a,b], [a,b)$ and $(a,b]$ as subsets of $\overline{\cZ}$.
Each point $z \in \cZ$ has an immediate successor $z^+$ and an immediate predecessor $z^-$, i.e.\ points $z^+,z^- \in \cZ$ such that $z^- < z < z^+$ and $(z^-,z) = (z, z^+) = \varnothing$, see Figure \ref{fig:admissible subset}.

\begin{notation}
	For a point $z \in \cZ$ we iteratively set, for $m \geq 0$:
	\[
		z^{(0)} = z; \; \; z^{(m+1)} = (z^{(m)})^+; \; \; z^{(-m-1)} = (z^{(-m)})^-.
	\]
\end{notation}

\begin{definition}

An {\bf arc} of $\cZ$ is a $2$-element subset $\{z_0,z_1\} \subseteq \cZ$ such that $z_1 \notin \{z_0^-,z_0,z_0^+\}$. We call $z_0$ and $z_1$ the {\bf endpoints} of the arc $\{z_0,z_1\}$.

\end{definition}

Igusa and Todorov \cite[Section~2.4]{IT} define the {\bf discrete cluster category} $\cC(\cZ)$ --  a $\bK$-linear, triangulated, 2-Calabi-Yau, Krull-Schmidt category associated to $\cZ$. Its nontrivial indecomposable objects correspond to arcs of $\cZ$. 
If $\{y_0,y_1\}$ is an arc of $\cZ$, by abuse of notation we also write $\{y_0,y_1\}$ for the corresponding indecomposable object. Suspension $\Sigma$ on $\cC(\cZ)$ acts on indecomposable objects as 
\[
	\Sigma(\{y_0,y_1\}) = \{y_0^-,y_1^-\}.
\]
If $\{y_0,y_1\}$ and $\{y_0',y_1'\}$ are arcs of $\cZ$ with
			\[
				y_0 < y_0' < y_1 < y_1' \hspace{10pt} \text{ or } \hspace{10pt} y_0 < y_1' < y_1 < y_0',
			\]
then we say that $\{y_0,y_1\}$ and $\{y_0',y_1'\}$ {\bf cross}.	
The Hom-spaces in the category $\cC(\cZ)$ can be read off from the respective positioning of the arcs and are described as follows:
$$\Hom_{\cC(\cZ)}(\{y_0,y_1\},\Sigma \{y_0',y_1'\})=\begin{cases} \bK, \text{ if } \{y_0,y_1\} \text{ and } \{y_0',y_1'\} \text{ cross, }\\ 0, \text{ otherwise. } \end{cases}$$
A non-zero morphism from $\{y_0,y_1\}$ to $\{y'_0,y'_1\}$ with $y_0 < y'^+_0 < y_1 < y_1'^+$ factors through an indecomposable object $S$ if and only if $S$ is of the form $\{s_0,s_1\}$ with $y_0 \leq s_0 \leq y'_0$ and $y_1 \leq s_1 \leq y'_1$, cf.\ \cite[Lemma 2.4.2]{IT}.
In the case when the arcs $\{y_0,y_1\}$ and $\{y'_0,y_1'\}$ cross, say with $y_0 < y_0' < y_1 < y_1'$, we have the following triangle for every non-zero map $f\in \Hom_{\cC(\cZ)}(\{y_0,y_1\},\Sigma \{y_0',y_1'\})$:
\begin{equation}\label{ExtOfArcs}
    \{y_0',y_1'\} \rightarrow X\oplus Z \rightarrow \{y_0,y_1\} \xrightarrow{f} \Sigma \{y_0',y_1'\},
\end{equation}  
where $X$ and $Z$ are the objects corresponding to the arcs $\{y_0,y_0'\}$ and $\{y_1,y_1'\}$. Note that the morphisms $\{y_0',y_1'\} \rightarrow X$, $\{y_0',y_1'\} \rightarrow Z$, $X \rightarrow \{y_0,y_1\}$ and $ Z \rightarrow \{y_0,y_1\}$ appearing in the triangle \eqref{ExtOfArcs} are all nonzero, if the corresponding objects are non-zero.

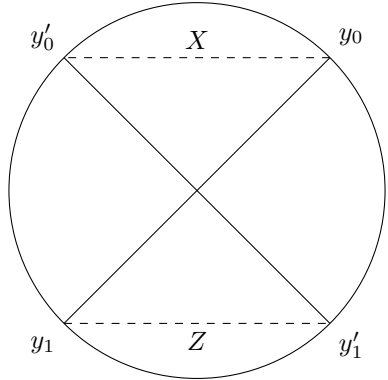
\begin{figure}[ht]
\begin{center}
\begin{tikzpicture}[scale=5]
        \draw (0,0) circle(0.5cm);
        
 	\node at (45:0.58){$y_0$};
 	\node at (135:0.58){$y'_0$};
 	\node at (225:0.58){$y_1$};
 	\node at (315:0.58){$y'_1$};
 	
 		\draw (45:0.5) -- (225:0.5);
 		\draw (135:0.5) -- (315:0.5);
 		
 		\draw[dashed] (45:0.5) -- (135:0.5);
 		\draw[dashed] (225:0.5) -- (315:0.5);
 		
 		\node at (90:0.4){$X$};
 		\node at (270:0.4){$Z$};

  \end{tikzpicture}
  \end{center}
  \caption{An illustration of the cocone $X\oplus Z$ of a map corresponding to the intersection of the arcs $\{y_0,y_1\}$ and $\{y'_0,y'_1\}$.
\label{fig:cocone}}
\end{figure}
\section{Thick subcategories}\label{S:thicksubcats}

In this section we classify thick subcategories of $\cC(\cZ)$ and describe the lattice of thick subcategories using non-exhaustive non-crossing partitions of $[n]$, where $[n]$ is the poset $[n] = \{1, \ldots, n\}$.

\begin{definition}\label{DefThick}
    Let $\cT$ be a triangulated category. A full subcategory $\cX$ of $\cT$ is called {\bf thick} if it is closed under cones,  direct summands and the action of $\Sigma$ and $\Sigma^{-1}$, i.e.\ it is a triangulated subcategory of $\cT$ closed under direct summands. 
\end{definition}

The set of thick subcategories of an essentially small triangulated category $\cT$ forms a complete lattice under inclusion, that is a poset with arbitrary joins and meets. Indeed, since thick subcategories are defined by closure conditions (see Definition \ref{DefThick}), an arbitrary intersection of thick subcategories is still a thick subcategory and so arbitrary meets exist. The existence of arbitrary joins follows automatically. Namely, the join of a set of thick subcategories can be given as the intersection of all thick subcategories containing each subcategory from the set. We denote the lattice of thick subcategories of $\cT$ by $\mathrm{Thick}(\cT)$.
For a set of arcs $\cY$ we will denote by $\add (\cY)$  the smallest subcategory containing the arcs from $\cY$ and closed under direct sums and summands. Mapping a set of arcs $\cY$ to $\add (\cY)$ yields a bijection between sets of arcs of $\cZ$ and subcategories of $\cC(\cZ)$. In particular, the empty set of arcs corresponds to the subcategory containing only $0$. By abuse of notation we will sometimes write $\cY$ instead of $\add (\cY)$.

The following lemma describes a couple of elementary features of triangulated categories, see for example Neeman's exposition in \cite[Chapter~1]{Neemanbook}.

\begin{lemma}\label{ConeOfIso}\label{ConeOfZero}
Let $ X$ and $Y_j$, $j=1,\dots,m$ be objects in $\cC(\cZ)$ and let $\Sigma^{-1}X\xrightarrow{f} \bigoplus\limits_{j=1}^mY_j$ be a morphism with components $f_j: \Sigma^{-1}X\rightarrow Y_j$ for $1\leq j\leq m$. Consider a triangle 
$$\Sigma^{-1}X\xrightarrow{f} \bigoplus_{j=1}^mY_j\rightarrow C \rightarrow X. $$
\begin{enumerate}
    \item If $f_k$ is an isomorphism for some $1\leq k \leq m$, then $C\simeq \bigoplus\limits_{j=1,j\neq k}^mY_j$.
\item If $f_k$ is zero for some $1\leq k \leq m$, then $C\simeq Y_k\oplus C'$, where $C'$ is given by the triangle:
$$\Sigma^{-1}X\xrightarrow{f'} \bigoplus_{j=1, j\neq k}^mY_j\rightarrow C' \rightarrow X, \text{ where } f' \text{ has components } f'_j=f_j \text{ for } j\neq k.$$
\end{enumerate}
\end{lemma}

\begin{lemma}\label{ConeInd}
Let $ X$ and $Y_j$, $j=1,\dots,m$ be indecomposable objects in $\cC(\cZ)$, corresponding to arcs with endpoints in some set $\cX\subseteq \cZ$. Then for any triangle of the form
\begin{equation}\label{E:ConeEndpoints}
    \Sigma^{-1}X\xrightarrow{f} \bigoplus_{j=1}^mY_j\rightarrow C \rightarrow X
\end{equation}
the object $C$ decomposes as a direct sum of indecomposable objects represented by arcs with endpoints in $\cX$.

Dually, for indecomposable objects $X_i$, $i=1,\dots,l$ and $Y$ in $\cC(\cZ)$, corresponding to arcs with endpoints in some set $\cX\subseteq \cZ$, any object $C$ arising from a triangle of the form
$$\Sigma^{-1}\bigoplus_{i=1}^lX_i\rightarrow Y\rightarrow C\rightarrow  \bigoplus_{i=1}^lX_i $$
 decomposes as a direct sum of indecomposable objects represented by arcs with endpoints in $\cX$.
\end{lemma}

\begin{proof}
We will prove only the first statement, the second statement follows analogously. The statement holds for $m=1$ using triangle \eqref{ExtOfArcs} in case the map $\Sigma^{-1}X\rightarrow Y$ is non-zero and by Lemma \ref{ConeOfZero} (2) otherwise. In case $f$ is an isomorphism we have $C=0$. We prove the statement by induction on $m$. 

Let us consider the triangle \eqref{E:ConeEndpoints}.
By Lemma \ref{ConeOfIso} above we can assume without loss of generality that the components $f_j: \Sigma^{-1}X\rightarrow Y_j$ of $f$ for $1\leq j\leq m$  are neither isomorphisms nor zero. 
Let $X=\{x,x'\}$ and let $Y_j = \{y_j,y'_j\}$. We can pick $k \in \{1, \ldots, m\}$ in such a way that $x< y_k<x'$ and there are no $y=y_j$ or $y=y'_j$ for $j=1, \ldots, m$ with $x<y<y_k$. The arc $Y_k$ is not necessarily unique, we fix one such arc. Renumbering the arcs, we can assume $k=1$.
 Consider the following octahedral diagram, where  $\pi_1$ is the projection on $Y_k=Y_1$:

$$ 
 \xymatrix{
&\bigoplus\limits_{j=2}^mY_j \ar@{=}[r] \ar[d] &\bigoplus\limits_{j=2}^mY_j \ar[d] &\\
\Sigma^{-1}X\ar[r] \ar@{=}[d] &\bigoplus\limits_{j=1}^mY_j \ar[r] \ar[d]^{\pi_1}&C \ar[r] \ar[d]& X  \ar@{=}[d]\\
\Sigma^{-1}X \ar[r]^{f_1}&Y_1 \ar[r] \ar[d]&C' \ar[r] \ar[d]&  X. \\
& \Sigma\bigoplus\limits_{j=2}^m Y_j \ar@{=}[r]&\Sigma\bigoplus\limits_{j=2}^m Y_j&
}
$$
In case $y_k=x^+$ or $y'_k={x'}^+$ the object $C'$ is indecomposable, $C'=\{x',y'_k\}$ or $C'=\{x,y_k\}$, respectively. We obtain the statement applying the induction hypothesis to the triangle 
$$\Sigma^{-1}C'\rightarrow \bigoplus_{j=2}^mY_j \rightarrow C \rightarrow  C'. $$

Otherwise, we get $C'=C_1\oplus C_2$, where $C_1=\{x,y_k\}$ and $C_2=\{x',y'_k\}$. By assumption $C_1$ does not intersect any of the arcs $Y_2,\dots,Y_m$. This implies that the triangle 
$$\bigoplus_{j=2}^mY_j \rightarrow C \rightarrow C_1\oplus C_2 \rightarrow \Sigma \bigoplus_{j=2}^mY_j$$
decomposes as a direct sum of triangles

$$0 \rightarrow C_1 \rightarrow C_1 \rightarrow 0 \hspace{10pt} \text{ and } \hspace{10pt} \bigoplus_{j=2}^mY_j \rightarrow C'' \rightarrow  C_2 \rightarrow \Sigma \bigoplus_{j=2}^mY_j. $$
Applying the induction hypothesis to the last triangle we obtain the statement.
\end{proof}

\begin{lemma}\label{ThickLem}
Let $X_i$, $i=1,\dots,l$ and $Y_j$, $j=1,\dots,m$ be indecomposable objects in  $\cC(\cZ)$, corresponding to arcs with endpoints in some set $\cX \subset \cZ$. Then any object $C$ appearing in a triangle of the form
$$\Sigma^{-1}\bigoplus_{i=1}^lX_i\rightarrow \bigoplus_{j=1}^mY_j\rightarrow C \rightarrow \bigoplus_{i=1}^lX_i $$
decomposes as a direct sum of indecomposable objects, corresponding to arcs with endpoints in $\cX$.
\end{lemma}

\begin{proof}
We will prove the statement by induction again. Let $l$ be fixed, we will do the induction on $m$. Lemma \ref{ConeInd} is the base of induction. For the general case consider the following octahedral diagram
\[
\xymatrix{
&Y_1 \ar@{=}[r] \ar[d]^{\iota_1} &Y_1 \ar[d] &\\
\Sigma^{-1}\bigoplus\limits_{i=1}^lX_i\ar[r] \ar@{=}[d] &\bigoplus\limits_{j=1}^mY_j \ar[r] \ar[d]&C \ar[r] \ar[d]& \bigoplus\limits_{i=1}^lX_i  \ar@{=}[d]\\
\Sigma^{-1}\bigoplus\limits_{i=1}^lX_i \ar[r]&\bigoplus\limits_{j=2}^mY_j \ar[r] \ar[d]&C' \ar[r] \ar[d]&  \bigoplus\limits_{i=1}^lX_i, \\
& \Sigma Y_1 \ar@{=}[r]&\Sigma Y_1&
}
\]
where $\iota_1$ is the inclusion of $Y_1$ into $\bigoplus\limits_{j=1}^mY_j$ as a direct summand. By induction hypothesis $C'$ is a sum of indecomposable objects with endpoints in $\cX$. By Lemma \ref{ConeInd} so is $C$.
\end{proof}

\begin{definition}\label{DefNonEx}
	A {\bf non-exhaustive non-crossing partition of $\mathbf{[n]}$} is a collection $\cP = \{B_m \subseteq [n] \mid m \in I\}$ of non-empty subsets of $[n]$,
	for some indexing set $I$, such that the following conditions hold. 
	Blocks $B_{m}$ do not intersect, i.e.\ $B_{m_1}\cap B_{m_2}=\varnothing$ for $m_1\neq m_2\in I$. Furthermore whenever $i,j,k,l \in [n]$ with 
	\[
		1 \leq i < k < j < l \leq n
	\] 
	and $i,j \in B_{m_1}$, $k,l \in B_{m_2}$ for some $m_1, m_2 \in I$, then we must have $m_1 = m_2$. 
\end{definition}	
	
	Note that a collection $\cP$ in Definition \ref{DefNonEx} may be empty. The notion of a non-exhaustive non-crossing partition is a generalization of the well-known notion of a non-crossing partition. Namely, a {\bf non-crossing partition of $\mathbf{[n]}$} can be defined as a non-exhaustive non-crossing partition $\cP = \{B_m \mid m \in I\}$ covering the whole set $[n]$, that is  $\cup_{m\in I}B_m=[n]$. A non-exhaustive non-crossing partition of $[n]$ can then be viewed as a non-crossing partition of a subposet of $[n]$.

Consider a regular $n$-gon with vertices labelled consecutively by $1, \ldots, n$. A non-exhaustive non-crossing partition of $[n]$ gives a partition of a subset of  the  vertices of this $n$-gon into blocks  with the property that the convex hulls of the blocks are pairwise disjoint. We will denote the set of all non-exhaustive non-crossing partition of $[n]$ by $NNC_n$. 

\begin{remark}
The number of non-exhaustive non-crossing partition of $[n]$ can be computed by the formula 
\[
    |NNC_n|=\sum_{k=0}^n \genfrac(){0pt}{0}{n}{k} \cdot C_k, 
\]
where $C_k$ denotes the Catalan number.
\end{remark}

Let us fix two non-exhaustive non-crossing partitions $\cP = \{B_m \mid m \in I\}$ and $\cP' = \{B'_{m'} \mid m' \in I'\}$ in $NNC_n$. The set $NNC_n$ forms a lattice under the following relation: We set $\cP\leq \cP'$ if for any block $B_m\in \cP$ there exists a block $B'_{m'} \in \cP'$ such that $B_m\subseteq B'_{m'}$. The fact that this relation gives a partial order on $NNC_n$ is straightforward. The meet operation for $\cP$ and $\cP'$ is given as follows: 
\[
    \cP \wedge \cP' = \{  B_m \cap B'_{m'} \mid m \in I, m' \in I' \text{ and } B_m \cap B'_{m'}\neq \varnothing\}.
\]
The poset $NNC_n$ has a
least element - the empty collection of blocks - and a greatest element - the collection consisting of one block $[n]$, which implies by definition that $NNC_n$ is bounded. Since $NNC_n$ is finite it automatically has a join operation, as any finite, bounded meet semi-lattice (i.e.\ a poset where all pairs of elements have a meet). The join can be described as follows: Given $\cP \in NNC_n$ we can uniquely extend it to the non-crossing partition
\[
    \overline{\cP} = \cP \cup \{\{k\} \mid \text{ there exists no } B \in \cP \text{ such that } k \in B\}
\]
by adding singletons. It is straightforward to check that the join of $\cP_1$ and $\cP_2$ in $NNC_n$  is given by
\[
    \cP_1 \vee \cP_2 = (\overline{\cP}_1 \vee \overline{\cP}_2) \setminus \{\{k\} \mid \text{ there exists no } B \in \cP_1 \cup \cP_2 \text{ such that }k \in B\},
\]  
where $\overline{\cP}_1 \vee \overline{\cP}_2$ is the join of non-crossing partitions as described by Kreweras \cite{K}.

\begin{notation}
Let $\cP = \{B_m \mid m \in I\}$ be a non-exhaustive non-crossing partition of $[n]$. We consider the following full subcategory of $\cC(\cZ)$ closed under direct sums and summands and containing the $0$-object:
\[
	\langle \cP\rangle = \add \{\{y_0,y_1\} \; \text{arc of}\; \cZ \mid y_0,y_1 \in \bigcup_{i \in B_m} (a_i,a_{i+1}) \; \text{for some} \; m \in I\}.
\]    
Recall that $a_1, \ldots, a_n$ denote the limit points of $\cZ$ in an anti-clockwise order.
\end{notation}

\begin{thm}\label{T:thick}
Let $\cC(\cZ)$ be the discrete cluster category associated to $\cZ$. There is an isomorphism of lattices 
\[
		NNC_n \cong \mathrm{Thick}(\cC(\cZ)).
	\]
Under this isomorphism a non-exhaustive non-crossing partition $\cP$ corresponds to the subcategory $\langle \cP\rangle$.
\end{thm}

\begin{proof}

We will label the interval $(a_i,a_{i+1})$ of $\cZ$  by $i \in [n]$. First, let us check that the assignment in the statement above is well-defined. For a non-exhaustive non-crossing partition $\cP=\{B_m\mid m\in I\} \in NNC_n$ the corresponding  full subcategory $\langle \cP\rangle$ of $\cZ$ is closed under direct sums, direct summands, shifts and extensions by Lemma \ref{ThickLem}, and hence is a thick subcategory of $\cC(\cZ)$.

Let us now take a thick subcategory $\cX$ of $\cC(\cZ)$ and let $\cP_\cX=\{B_m\mid m\in I\} \in NNC_n$ be the non-exhaustive non-crossing partition of $[n]$ constructed from an equivalence relation on a subset $S$ of $[n]$ as follows. An element  $i$ is in $S$ if and only if there exists an arc $\{x,y\}\in \cX$ with $x\in (a_i,a_{i+1})$. Two elements $i$ and $j$ are equivalent if and only if there exists an arc $\{x,y\}\in \cX$ with $x\in (a_i,a_{i+1})$ and $y\in (a_j,a_{j+1})$. Note that we allow for $i = j$. 

We need to check that this assignment gives an equivalence relation. Symmetry is clear. Let's check reflexively. Let $j \in S$.
Then $\cX$ contains an arc $\{x,y\}$ with $x\in (a_i,a_{i+1})$, $y\in (a_j,a_{j+1})$, for some $i \in S$. Considering extensions between $\{x,y\}$ and $\Sigma\{x,y\}=\{x^-,y^-\}$, we get that $\{x^-,y\}$, $\{x,y^-\}$ and $\Sigma^{-}\{x^-,y\}=\{x,y^+\}$ belong to $\cX$, thus any arc of the form $\{x,y'\}$ with $y'\in (a_j,a_{j+1})$ belongs to $\cX$. Considering extensions between two crossing arcs $\{x,y'\}$ and $\{x^-,y''\}$, we get that $\{y',y''\}$ is in $\cX$ for any $y',y''\in (a_j,a_{j+1})$, and hence $j \sim j$. 
Next we check transitivity. If we have arcs $\{x,y\}\in \cX$ with $x\in (a_i,a_{i+1})$, $y\in (a_j,a_{j+1})$ and $\{y',z'\}\in \cX$ with $y'\in (a_j,a_{j+1})$ and $z\in (a_k,a_{k+1})$ with $i,j,k$ distinct, then shifting one of the arcs we can assume that $\{x,y\}$ and $\{y',z'\}$ cross. Thus one of the extensions between the two arcs contains the arc $\{x,z'\}$  and $i\sim j$, $j\sim k$ implies $i \sim k$.

The assignment gives a non-exhaustive non-crossing partition of $[n]$. Indeed, if we can find $i,j,k,l \in [n]$ with 
	$
		1 \leq i < k < j < l \leq n
	$ 
	and $i,j \in B_{m_1}$ and $k,l \in B_{m_2}$ for some $m_1, m_2 \in I$, then there are arcs $\{x,y\}\in \cX$ with $x\in (a_i,a_{i+1})$, $y\in (a_j,a_{j+1})$ and $\{z,w\}\in \cX$ with $z\in (a_k,a_{k+1})$, $w\in (a_l,a_{l+1})$. The arcs necessarily cross, so $\cX$ contains the arcs $\{x,z\}$, $\{z,y\}$, $\{y,w\}$ and $\{w,x\}$, which implies that $m_1 = m_2$.
	
Let us check that the two assignments are mutually inverse. The composition $\cP \mapsto \langle \cP \rangle \mapsto \cP_{\langle \cP \rangle} $	clearly gives the identity.
Starting from a thick subcategory $\cX$, the inclusion $\cX\subseteq \langle \cP_\cX \rangle$ is also clear. To get the other inclusion, we need to check that if $\cX$ contains an arc $\{x,y\}$ with $x\in (a_i,a_{i+1})$, $y\in (a_j,a_{j+1})$, then it contains all the arcs of the form $\{x',y'\}$ with $x'\in (a_i,a_{i+1})$, $y'\in (a_j,a_{j+1})$. We have already seen that $\cX$ contains any arc of the form $\{x,y''\}$ with $y''\in (a_j,a_{j+1})$. We can shift $\{x',y'\}$ in such a way that $\Sigma^n\{x',y'\}=\{x,y''\}\in\cX$, which implies that $\{x',y'\}\in \cX$, as $\cX$ is closed under shift. 

Both assignments are order preserving, and hence we have constructed an isomorphism of lattices.
\end{proof}

\begin{figure}[ht]
\begin{center}
\begin{tikzpicture}[scale=5]
	
	
	
	
	
	
	\fill[fill=gray!30]    (1:0.5) arc (1:59:0.5) -- (121:0.5) arc(120:179:0.5) -- (1:0.5);
	
	\fill[fill=gray!30]    (181:0.5) arc (181:359:0.5) -- (181:0.5);

	
	
        \draw (0,0) circle(0.5cm);
        
	\draw (0:0.5) node[fill=white,circle,inner sep=0.065cm] {} circle (0.02cm);
	\node at (0:0.58){$a_1$};
	\draw (60:0.5) node[fill=white,circle,inner sep=0.065cm] {} circle (0.02cm);
	\node at (60:0.58){$a_2$};
	\draw (120:0.5) node[fill=white,circle,inner sep=0.065cm] {} circle (0.02cm);
	\node at (120:0.6){$a_3$};
	\draw (180:0.5) node[fill=white,circle,inner sep=0.065cm] {} circle (0.02cm);
	\node[] (a_4) at (180:0.58){$a_4$};	
	\draw (240:0.5) node[fill=white,circle,inner sep=0.065cm] {} circle (0.02cm);
	\node[] (a_5) at (240:0.58){$a_5$};	
	\draw (300:0.5) node[fill=white,circle,inner sep=0.065cm] {} circle (0.02cm);
	\node[] (a_6) at (300:0.58){$a_6$};	
	
	
	\node at (30:0.58){$1$};
	\node at (90:0.58){$2$};
	\node at (150:0.58){$3$};
	\node at (210:0.58){$4$};
	\node at (270:0.58){$5$};
	\node at (330:0.58){$6$};

	\draw[dotted] (30:0.5) -- (150:0.5);
	
	\draw[dotted] (210:0.5) -- (270:0.5);
	\draw[dotted] (270:0.5) -- (330:0.5);
	\draw[dotted] (210:0.5) -- (330:0.5);

  \end{tikzpicture}
  \end{center}
  \caption{An illustration of the thick subcategory associated to the non-exhaustive non-crossing partition $\{\{1,3\},\{4,5,6\}\}$ of $[6]$. Its indecomposable objects correspond to arcs which, when drawn as straight lines between their endpoints, lie completely in one of the two grey shaded regions.}
\label{fig:thick}
\end{figure}
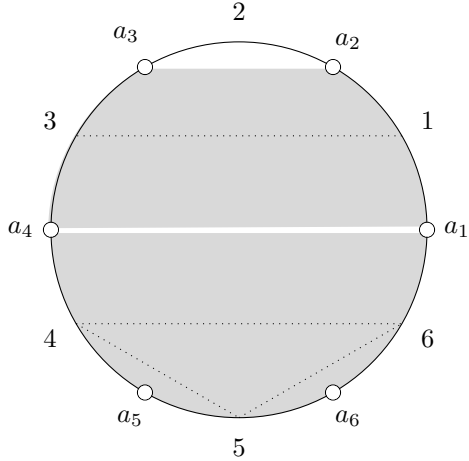

\begin{remark}
    Note that we could have avoided Lemma \ref{ThickLem} in the proof of the fact that given a non-exhaustive non-crossing partition $\cP$ the subcategory $\langle \cP \rangle$ is thick. Indeed, if $\cP$ consists of a single block $B$, then $\langle \cP \rangle$ is thick as it is a right or left orthogonal of a collection of arcs (one can take all arcs, which do not intersect arcs from $\langle \cP \rangle$). For a general $\cP = \{B_m \mid m \in I\}$ we get that $\langle \cP \rangle$ is a union of mutually orthogonal thick subcategories of the form $\langle \{B_m\} \rangle$ and hence is a thick subcategory as well. 
\end{remark}

\section{T-structures and non-crossing partitions}

In this section, we classify t-structures in $\cC(\cZ)$ via decorated non-crossing partitions. Let us start with some definitions.

\begin{definition}
	Let $\cT$ be a triangulated category. A {\bf torsion pair} in $\cT$ is a pair of subcategories $(\cX,\cY)$ of $\cT$ closed under direct summands and such that the following two conditions hold.
	For any $X\in \cX$ and $Y\in\cY$ we have $\Hom_\cT(X,Y) = 0$.
		For every object $T \in \cT$ there exists a distinguished triangle
		\begin{equation}\label{ApproxTri}
		    X \to T \to Y \to \Sigma X,
		\end{equation}
		with $X \in \cX$ and $Y \in \cY$.
		If $(\cX,\cY)$ is a torsion pair, we call $\cX$ its {\bf torsion class} and $\cY$ its {\bf torsion free class}.
	
	A {\bf t-structure} is a torsion pair $(\cX,\cY)$ such that $\cX$ is closed under suspension. If $(\cX,\cY)$ is a t-structure, we call its torsion class $\cX$ its {\bf aisle} and its torsion free class $\cY$ its {\bf coaisle}. In case of t-structures, we will call the triangle \eqref{ApproxTri} the {\bf approximation triangle} of $T$. Note that in that case any two triangles of the form $X \to T \to Y \to \Sigma X,$
with $X \in \cX$ and $Y \in \cY$ are isomorphic.
\end{definition}

\begin{remark}
	A torsion pair $(\cX,\cY)$ in a triangulated category $\cT$ is uniquely defined by its torsion class (or, equivalently, its torsion free class). In fact, if $(\cX,\cY)$ is a torsion pair, then 
	\[
		\cY = \cX^{\perp} = \{T \in \cT \mid \Hom_\cT(X,T) = 0 \text{ for all objects $X\in \cX$}\}
	\]
	and
	\[
		\cX = \;^{\perp}\cY = \{T \in \cT \mid \Hom_\cT(T,Y) = 0 \text{ for all objects $Y \in \cY$}\}.
	\]
	In particular, in order to classify t-structures, it is sufficient to classify aisles of t-structures (or, equivalently, coaisles of t-structures).
\end{remark}
Torsion classes in the category $\cC(\cZ)$ were described in \cite{GHJ} using combinatorial conditions PC and PTO. The abbreviation PC stands for ``precovering'', and PTO stands for "Ptolemy". We will use the following notation to describe these conditions.

\begin{notation}
	If $(x_i)_{i \geq 0}$ is a sequence from $\cZ$ converging to $a \in \cZ$, we write $x_i \to a$.
\end{notation}

If $(x_i)_{i \geq 0}$ is a sequence from $\cZ$ with $x_i \to a$, and there exists $z \in \cZ$ such that $x_i \in [z,a]$ for large enough $i$, we say that $x_i$ is a sequence converging to $a$ from below and write $x_i \to a$ from below. Symmetrically, if there exists $z \in \cZ$ such that $x_i \in [a,z]$ for large enough $i$, we say that $x_i$ is a sequence converging to $a$ from above and write $x_i \to a$ from above. 

\begin{itemize}
	\item[PC] Whenever $(\{y_0^i,y_1^i\})_{i \geq 0}$ is a sequence of arcs from $\cY$ such that $y_0^i \to a$ from below, and $y_1^i \to b$ with $a \neq b$, then there exists a sequence $(\{z_0^i,z_1^i\})_{i \geq 0}$ from $\cY$ with $z_0^i \to a$ from above and $z_1^i \to b$ from above.
	\item[PTO] Whenever $\{y_0,y_1\} \in \cY$ and $\{y_0',y_1'\} \in \cY$ cross, then all arcs of $\cZ$ with both endpoints in $\{y_0,y_0',y_1,y_1'\}$ also lie in $\cY$. 
\end{itemize}

Figure \ref{fig:b_1 in Z} provides an illustration for PC when $a \in L(\cZ)$ and $b \in \cZ$.

		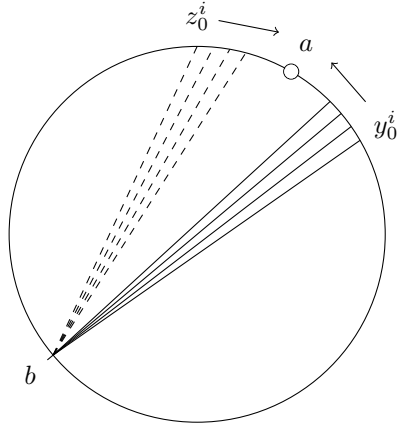
\begin{figure}[ht]
\begin{center}
\begin{tikzpicture}[scale=5]
        \draw (0,0) circle(0.5cm);
        
	\draw (60:0.5) node[fill=white,circle,inner sep=0.065cm] {} circle (0.02cm);
	\node[] (a_k) at (60:0.58){$a$};

	\node at (220:0.58){$b$};
	\draw  (220:0.48) -- (220:0.52);
	
	

	
	\node (z_0) at (90:0.58){$z_0^i$};
	\draw[->] (z_0) -- (68:0.58);
	
	\draw[dashed] (220:0.5) -- (90:0.5);
	\draw[dashed] (220:0.5) -- (85:0.5);
	\draw[dashed] (220:0.5) -- (80:0.5);
	\draw[dashed] (220:0.5) -- (75:0.5);
	
	\node (y_0) at (30:0.58){$y_0^i$};
	\draw[->] (y_0) -- (52:0.58);
	
	\draw (220:0.5) -- (30:0.5);
	\draw (220:0.5) -- (35:0.5);
	\draw (220:0.5) -- (40:0.5);
	\draw (220:0.5) -- (45:0.5);
	
  \end{tikzpicture}
  \end{center}
  \caption{An illustration of PC when $a \in L(\cZ)$ and $b \in \cZ$.
\label{fig:b_1 in Z}}
\end{figure}

\begin{thm}[\cite{GHJ}]\label{T:torsion pairs}
 There is a one-to-one correspondence between torsion classes in $\cC(\cZ)$ and collections of arcs of $\cZ$ satisfying PC and PTO. It sends a torsion class $\cX$ to the collection of arcs corresponding to its indecomposable objects.
\end{thm}

We say that a set of arcs $\cX$ of $\cZ$ is {\bf closed under clockwise rotation}, if, whenever $\{x,y\} \in \cX$ then $\{x^-,y^-\} \in \cX$. It is immediate that the one-to-one correspondence from Theorem \ref{T:torsion pairs} restricts to the following one-to-one correspondence:

\[
  {\left\{ \text{$\cX$ set of arc of $\cZ$}\ \middle\vert \begin{array}{l}
    \cX \; \text{is closed under clockwise rotation and} \\
    \text{satisfies PC and PTO}
  \end{array}\right\} \overset{1:1}{\longleftrightarrow} {\begin{Bmatrix}\text{aisles of t-structures of $\cC(\cZ)$} \end{Bmatrix}} }
\]

We will show that each t-structure on $\cC(\cZ)$ is associated to a unique non-crossing partition of $[n]$.

\begin{definition}
   Let $\cP$ be a (non-crossing) partition of $[n]$. 
    An element $i \in [n]$ is called a {\bf singleton of $\cP$} if $\{i\}$ is a block of $\cP$.
    An element $i \in [n]$ is called an {\bf adjacency of $\cP$} if $i$ and $i+1$ are in the same block of $\cP$ (where we think cyclically modulo $n$ and set $n+1 = 1$).
\end{definition}

To completely classify t-structures in terms of non-crossing partitions, we need the notion of a $\overline{\cZ}$-decoration of $\cP$. Recall that $\overline{\cZ}$ is the topological closure of $\cZ$.

\begin{definition}\label{D:Zdecoratedncpartition}
Let $\cP$ be a non-crossing partition of $[n]$. A {\bf $\overline{\cZ}$-decoration of $\cP$} is an $n$-tuple $(x_1, \ldots, x_n)$ such that
    \[
		x_i \in \begin{cases}
					[a_i,a_{i+1}) & \text{if $i$ is a singleton of $\cP$} \\
					(a_i,a_{i+1}] & \text{if $i$ is an adjacency of $\cP$} \\
					(a_i,a_{i+1}) & \text{else.}
				\end{cases}
	\]
	 Recall that we consider the indices modulo $n$, so $a_{n+1} = a_1$. 
	 
	 A pair $(\cP,\bf{x})$ where $\cP$ is a non-crossing partition of $[n]$ and ${\bf x}$ is a $\overline{\cZ}$-decoration of $\cP$ is called a {\bf $\overline{\cZ}$-decorated non-crossing partition of $[n]$}.
\end{definition}

Note that depending on the partition $\cP$ some of the decorations $x_1, \ldots, x_n$ can be limit points. 

\subsection{Classification of t-structures}

Let $\cP = \{B_m \mid m \in I\}$ be a non-crossing partition of $[n]$, and let ${\bf x} = (x_1, \ldots, x_n)$ be a $\overline{\cZ}$-decoration of $\cP$. Similarly to the case of thick subcategories in the proof of Theorem \ref{T:thick} we set
\[
	\cX(\cP,{\bf x}) = \add \{\{y_0,y_1\} \; \text{arc of}\; \cZ \mid y_0,y_1 \in \bigcup_{i \in B_m} (a_i,x_i] \; \text{for some} \; m \in I\}.
\]

\begin{thm}\label{T:one-to-one}
	There is a one-to-one correspondence
	\[
		\{ \overline{\cZ}\text{-decorated non-crossing partitions of $[n]$}\} \to \{\text{t-structures of $\cC(\cZ)$}\}
	\]
	It sends a pair $(\cP, {\bf x})$ to the t-structure with the aisle $\cX(\cP,{\bf x})$.
\end{thm}

Figure \ref{fig:aisle} provides an illustration of an aisle for the case $n = 6$.

\begin{figure}[ht]
\begin{center}
\begin{tikzpicture}[scale=5]
	
	
	\fill[fill=gray!30]    (0:0.5) arc (0:35:0.5) -- (120:0.5) arc(120:155:0.5) -- (0:0.5);
	
	
	\fill[fill=gray!30]    (180:0.5) arc (180:200:0.5) -- (240:0.5) arc(240:335:0.5) -- (180:0.5);
	
        \draw (0,0) circle(0.5cm);
        
	\draw (0:0.5) node[fill=white,circle,inner sep=0.065cm] {} circle (0.02cm);
	\node at (0:0.58){$a_1$};
	\draw (60:0.5) node[fill=white,circle,inner sep=0.065cm] {} circle (0.02cm);
	\node at (60:0.58){$a_2 = x_2$};
	\draw (120:0.5) node[fill=white,circle,inner sep=0.065cm] {} circle (0.02cm);
	\node at (120:0.6){$a_3$};
	\draw (180:0.5) node[fill=white,circle,inner sep=0.065cm] {} circle (0.02cm);
	\node[] (a_4) at (180:0.58){$a_4$};	
	\draw (240:0.5) node[fill=white,circle,inner sep=0.065cm] {} circle (0.02cm);
	\node[] (a_5) at (240:0.58){$a_5$};	
	\draw (300:0.5) node[fill=white,circle,inner sep=0.065cm] {} circle (0.02cm);
	\node[] (a_6) at (300:0.58){$a_6 = x_5$};	
	
	\draw[dotted] (0:0.5) -- (120:0.5);
	
	\draw[dotted] (180:0.5) -- (240:0.5);
	\draw[dotted] (240:0.5) -- (300:0.5);
	\draw[dotted] (180:0.5) -- (300:0.5);

	
	\draw (35:0.52) -- (35:0.48);
	\node (x_1) at (35:0.58){$x_{1}$};
	
	\draw (200:0.52)  -- (200:0.48);
	\node (x_4) at (200:0.58){$x_{4}$};
	
	\draw (155:0.52)  -- (155:0.48);
	\node (x_{3}) at (155:0.58){$x_3$};
	
	\draw (335:0.52)  -- (335:0.48);
	\node (x_{6}) at (335:0.58){$x_6$};

  \end{tikzpicture}
  \end{center}
  \caption{
  An illustration of an aisle in the case $n = 6$. Its associated $\overline{\cZ}$-decorated non-crossing partition is $(\{\{1,3\},\{2\},\{4,5,6\}\}, (x_1,a_2,x_3,x_4,a_6,x_6))$, where for $i \in \{1,3, 4,6\}$ we have $x_i \in (a_i,a_{i+1})$. The indecomposable objects of the aisle correspond precisely to those arcs which, when drawn as straight lines, lie in the grey shaded regions 
}
\label{fig:aisle}
\end{figure}
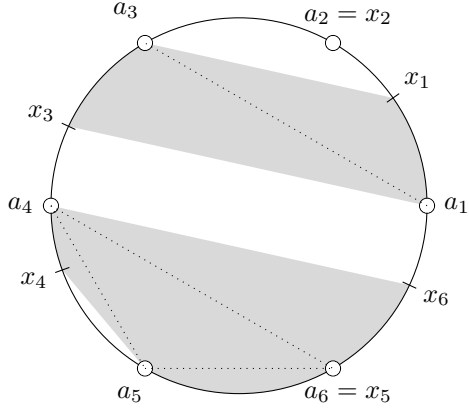

\begin{remark}\label{RemNg}
	For the special case $n=1$, t-structures have been classified by Ng in \cite{Ng}. In this case, the element $1$ should be interpreted as both a singleton and an adjacency in the unique partition of $[1]$. Ng's classification agrees with our classification from Theorem \ref{T:one-to-one}, and can be derived as a special case of our proof of Theorem \ref{T:one-to-one}. For ease of notation, since the result is already known for $n=1$, in what follows we will assume $n \geq 2$. This will allow us, in the very few cases where it is relevant, to avoid a cumbersome case distinction for the degenerate case $a_{i+1} = a_i$ for some $1 \leq i \leq n$, which occurs if and only if $n=1$.
\end{remark}

We use the following two propositions to show Theorem \ref{T:one-to-one}.

\begin{proposition}\label{P:injective}
	Let $(\cP, {\bf x})$ be a $\overline{\cZ}$-decorated non-crossing partition of $[n]$. Then $\cX(\cP,{\bf x})$ is the aisle of a t-structure.
\end{proposition}

\begin{proposition}\label{P:surjective}
	Let $\cX$ be the aisle of a t-structure in $\cC(\cZ)$. Then there exists a $\overline{\cZ}$-decorated non-crossing partition $(\cP, {\bf x})$ of $[n]$ such that $\cX = \cX(\cP,{\bf x})$.
\end{proposition}

Theorem \ref{T:one-to-one} is an immediate consequence of Propositions \ref{P:injective} and \ref{P:surjective} and the observation that if $(\cP,{\bf x}) \neq (\tilde{\cP}, \tilde{{\bf x}})$, then $\cX(\cP,{\bf x}) \neq \cX(\tilde{\cP},\tilde{{\bf x}})$.

\begin{proof}[Proof of Proposition \ref{P:injective}]
	Let $(\cP, {\bf x})$ be a $\overline{\cZ}$-decorated non-crossing partition of $[n]$ with $\cP = \{B_m \mid m \in I\}$ and ${\bf x} = (x_1, \ldots, x_n)$. Consider the set of arcs
	\[
		\cX = \{\{y_0,y_1\} \; \text{arc of}\; \cZ \mid y_0,y_1 \in \bigcup_{i \in B_m} (a_i,x_i] \; \text{for some} \; m \in I\}
	\] 
	corresponding to the subcategory $\cX(\cP,{\bf x})$.
	By construction, it is closed under clockwise rotation. It remains to show PTO and PC.
	
	{\bf PTO}: Assume that the arcs $\{y_0,y_1\} \in \cX$ and $\{y_0',y_1'\} \in \cX$	cross. 
	Since $\cP$ is a non-crossing partition we must have $\{y_0,y'_0,y_1,y'_1\} \subseteq \bigcup_{i \in B_m} (a_i,x_i]$ for some $m \in I$. In particular, all arcs with both endpoints in $\{y_0,y_1,y'_0,y_1'\}$ lie in $\cX$.

{\bf PC}: Let $(\{y_0^i,y_1^i\})_{i \geq 0}$ be a sequence from $\cX$ such that $y_0^i \to b_0$ from below and $y_1^i \to b_1$. We need to construct a sequence  $\{(z_0^i,z_1^i)\}_{i \geq 0}$ with  $z_0^i$ converging to $b_0$ from above and $z_1^i$  converging to $b_1$ from above.
We have $b_0 \in [a_{j},a_{j+1})$ and $b_1 \in [a_{k},a_{k+1})$ for some $j,k \in [n]$. If both $b_0$ and $b_1$ are in  $\cZ$ then we are done. Similarly, if $b_0 \in \cZ$ and there is a subsequence $\tilde{y}_1^i$ of $y_1^i$ converging to $b_1$ from above, then a subsequence of $\{b_0,\tilde{y}_1^i\}\subseteq \{y_0^i,y_1^i\}$ will serve as the desired sequence.
	
	Assume next that $b_0 = a_{j}$. Then we must have $x_{j-1} = a_{j}$, and so $j-1$ is an adjacency in $\cP$. 
	If there is a subsequence $\tilde{y}_1^i$ of $y_1^i$ converging to $b_1$ from above, then $k$ is in the same block as $j$ and $j-1$. 
	If there is a subsequence $\tilde{y}_1^i$ of $y_1^i$ converging to $b_1$ from below, then in case $b_1=a_k$ we get that $k$ is in the same block as $k-1$ and so  $k, k-1, j$ and $j-1$ are all in the same block; in case $b_1\in \cZ$ we get that $k, j$ and $j-1$ are in the same block as well. 
	
	We are left with the case where we have a sequence $(\{\tilde{y}_0^i,\tilde{y}_1^i\})_{i \geq 0}$ from $\cX$ such that $\tilde{y}_0^i \to b_0$ from below and $\tilde{y}_1^i \to b_1$ from below, $b_0\in \cZ$ and $b_1=a_k$. Changing the roles of $b_0$ and $b_1$ we can reduce to the case considered above and conclude that $k, k-1$ and $j$ are in the same block. In all the cases $j$ and $k$ are in the same block and are not singletons, so $(a_j,x_j) \neq \varnothing$ and $(a_k,x_k) \neq \varnothing$. We can pick a sequence $(\{z_0^i\})_{i \geq. 0}$ from $(a_j,x_{j})$ converging to $a_j$ from above and a sequence $(\{z_1^i\})_{i \geq 0}$ from $(a_k,x_{k})$ converging to $a_k$ from above, yielding the desired sequence $\{(z_0^i,z_1^i)\}_{i \geq 0}$.
			\end{proof}

To finish the proof of Theorem \ref{T:one-to-one} it remains to show Proposition \ref{P:surjective}. Let $\cX$ be the aisle of a t-structure in $\cC(\cZ)$. Note that the set of arcs corresponding to $\cX$ satisfies PC and PTO, and is closed under clockwise rotation. In order to prove Proposition \ref{P:surjective}, we explicitly construct a $\overline{\cZ}$-decorated non-crossing partition $(\cP(\cX),{\bf x}(\cX))$ with $\cP(\cX) = \{B_m \mid m \in I \}$ and  ${\bf x}(\cX) = (x_1, \ldots, x_n) $ (see Definition \ref{DefPPP} and Equation \eqref{EqDec}) such that
\[
	\cX = \add \{\{y_0,y_1\} \; \text{arc of}\; \cZ \mid y_0,y_1 \in \bigcup_{i \in B_m} (a_i,x_i] \; \text{for some} \; m \in I\}.
\] 
For $i \in [n]$ let us consider the set of all endpoints of arcs from $\cX$ contained in $(a_i,a_{i+1})$:
\[
	S_i = \{v \in (a_i,a_{i+1}) \mid \{v,w\} \in \cX \; \text{for some} \; w \in \cZ\}
\]
and
\begin{equation}\label{EqDec}
	x_i = \begin{cases}
			a_i & \text{if $S_i = \varnothing$};\\
			\max_{(a_i,a_{i+1})}S_i & \text{if $S_i \neq \varnothing$, and the maximum exists}; \\
			a_{i+1} & \text{if $S_i \neq \varnothing$, and $S_i$ does not have a maximum}.
		\end{cases}    
\end{equation}

Set ${\bf x}(\cX) = (x_1, \ldots, x_n)$. This will be the $\overline{\cZ}$-decoration for the non-crossing partition $\cP(\cX)$ we are going to construct via the following relation: For $i,j \in [n]$ set
\[
	i \sim_\cX j \Leftrightarrow \text{$i=j$ or there exists an arc} \; \{y_0,y_1\} \in \cX \; \text{such that} \; y_0 \in (a_i,x_i] \; \text{and} \; y_1 \in (a_j,x_j].
\]

\begin{lemma}
	The relation $\sim_{\cX}$ is an equivalence relation on $[n]$.
\end{lemma}

\begin{proof}
	Reflexivity and symmetry are clear, it remains to check transitivity. Assume thus that $i \sim_{\cX} j$ and $j \sim_{\cX} k$ and $i,j,k$ are pairwise distinct. Then there exist arcs $\{y_0,y_1\} \in \cX$ and $\{y'_0,y'_1\} \in \cX$ with $y_0,y'_0 \in (a_j,x_j]$, $y_1 \in (a_i,x_i]$ and $y'_1 \in (a_k,x_k]$. 
    By applying an appropriate power of the suspension $\Sigma$ to either $\{y'_0,y'_1\}$ or $\{y_0,y_1\}$, we can assume that $\{y_0,y_1\}$ and $\{y'_0,y'_1\}$ cross.
	So by PTO we have $\{y'_1, y_1\} \in \cX$, which yields $i \sim_{\cX} k$.
\end{proof}

\begin{definition}\label{DefPPP}
	We define $\cP(\cX)$ to be the partition of $[n]$ induced by the equivalence relation $\sim_{\cX}$, that is, its blocks are given by the equivalence classes under $\sim_{\cX}$.
\end{definition}

\begin{lemma}\label{L:decorated nc partition}
	The pair $(\cP(\cX), {\bf x}(\cX))$ is a $\overline{\cZ}$-decorated non-crossing partition of $[n]$.
\end{lemma}

\begin{proof}
Let us first prove that the partition $\cP(\cX)$ is non-crossing. Let $B_m$ and $B_{m'}$ be blocks  of $\cP(\cX)$ with $i,j \in B_m$ and $k,l \in B_{m'}$ such that $1 \leq i < k < j < l \leq n$. Thus there exist arcs $\{y_0,y_1\} \in \cX$ and $\{y'_0,y'_1\} \in \cX$ with
	\[
		y_0 \in (a_i,x_i], \; y_1\in (a_j,x_j], \; y'_0 \in (a_k,x_k], \; y'_1 \in (a_l,x_l].   
	\]
In particular the arcs $\{y_0,y_1\}$ and $\{y'_0,y'_1\}$ cross, so by PTO all arcs with both endpoints in $\{y_0,y'_0,y_1,y'_1\}$ lie in $\cX$. Therefore $B_m = B_{m'}$ and $\cP(\cX)$ is a non-crossing partition.

Let us now check that ${\bf x}(\cX)$ gives a $\overline{\cZ}$-decoration of $\cP$. All we need to show is that $x_i = a_i$ implies that $\{i\}$ is a singleton of $\cP(\cX)$ and if $x_i = a_{i+1}$ then $i$ is an adjacency of $\cP(\cX)$.
Indeed, if we have $x_i = a_i$, then $S_i = \varnothing$, so there exists no $\{v,w\} \in \cX$ with $v \in (a_i,a_{i+1})$. Thus $i \sim_{\cX} j$ if and only if $i =j$, and $\{i\}$ is a singleton of $\cP(\cX)$.

If we have $x_i = a_{i+1}$, then $S_i \neq \varnothing$ and the maximum of $S_i$ does not exist. Thus, since $\cX$ is closed under clockwise rotation, every $y \in (a_i,a_{i+1})$ is an endpoint of an arc in $\cX$.
We now show inductively that in fact for every $y \in (a_i,a_{i+1})$ and $m \geq 2$, the arc $\{y,y^{(m)}\}$ lies in $\cX$, thus constructing a sequence of arcs with endpoints $y^{(m)} \to a_{i+1}$ from below. For the base case $m = 2$ let $y \in (a_i,a_{i+1})$ and pick $z \in \cZ$ such that $\{y^{(2)}, z\} \in \cX$. If $z = y$ we are done. Else, by $\Sigma$-closure, the arc $\Sigma^2 \{y^{(2)}, z\} = \{y,z^{(-2)}\}$ also lies in $\cX$. If $z^{(-2)} = y^{(2)}$ we are done. In all other cases, the arcs $\{y,z^{(-2)}\}$ and $\{y^{(2)},z\}$ cross 
and by PTO we have $\{y,y^{(2)}\} \in \cX$ as desired. Assume now the assertion holds for all $2 \leq k < m$. By induction hypothesis, we have $\{y^+,y^{(m)}\} \in \cX$ and by $\Sigma$-closure the arc $\{y,y^{(m-1)}\}$ lies in $\cX$. Since these two arcs cross, by PTO we have $\{y,y^{(m)}\} \in \cX$ which concludes the induction.

Fix now $y \in (a_i,a_{i+1})$. Then $\{y,y^{(m)}\}_{m \geq 0}$ is a sequence of arcs from $\cX$ with $y^{(m)} \to a_{i+1}$ from below. By PC there exists a sequence of arcs $\{y,z^i\}_{i \geq 0}$ with $z^i \to a_{i+1}$ from above, without loss of generality we can assume $z^i \in (a_{i+1},x_{i+1}]$. Now $y \in (a_i,a_{i+1}) = (a_i,x_i]$, and so $i$ is an adjacency.
\end{proof}

\begin{proof}[Proof of Proposition \ref{P:surjective}]
	Let $\cX$ be the aisle of a t-structure in $\cC(\cZ)$. Consider the associated $\overline{\cZ}$-decorated non-crossing partition $(\cP(\cX), {\bf x}(\cX))$ with $\cP(\cX) = \{B_m \mid m \in I\}$ and ${\bf x} (\cX)= (x_1, \ldots, x_n)$. We show that $\cX = \cX(\cP(\cX),{\bf x}(\cX))$.
The inclusion $\cX \subseteq \cX(\cP(\cX), {\bf x}(\cX))$ is clear from the construction.
It remains to show $\cX(\cP(\cX), {\bf x}(\cX)) \subseteq \cX$. Let $\{y_0,y_1\} \in \cX(\cP(\cX), {\bf x}(\cX))$. So we have $y_0 \in (a_i,x_i]$ and $y_1 \in (a_j,x_j]$ for some $i \sim_{\cX} j$. We first prove that there is an arc $\{b_0,b_1\} \in \cX$ with $b_0 \in [y_0,x_i]$ and $b_1 \in [y_1, x_j]$. 

Assume first that $i = j$, and without loss of generality, $a_i < y_1 < y_0 \leq x_i$. By maximality of $x_i$, there exists an arc $\{b_0,c\} \in \cX$ with $b_0 \in [y_0,x_i]$. If $c \in (a_i,a_{i+1})$, then without loss of generality we may assume $c < b_0^{-2} \leq x_i$, else we set $b_1 = c$. Applying $\Sigma^2$ and using PTO we get that $\{b_0,b_0^{(-2)}\}\in\cX$ and set $b_1=b_0^{(-2)}$.

Assume now that $i \neq j$. By definition of $\sim_{\cX}$ there exists a $\{c_0,c_1\} \in \cX$ such that $c_0 \in (a_i,x_i]$ and $c_1 \in (a_j,x_j]$. We will first find $b_0 \in [y_0,x_i]$ such that $\{b_0,c_1\} \in \cX$. If we already have $c_0 \in [y_0,x_i]$, set $b_0=c_0$, else, take an arc $\{d_0,d_1\} \in \cX$ such that $d_0 \in [y_0,x_i] \subseteq (c_0,x_i]$. Such an arc exists by maximality of $x_i$. We distinguish four cases, depending on where the point $d_1$ is positioned.

Case 1: Assume $d_1 = c_1$. Setting $b_0 = d_0$ yields a point $b_0 \in [y_0,x_i]$, with $\{b_0,c_1\} \in \cX$.

Case 2: Assume $d_1 \in (c_1,c_0)$. Then $\{c_0,c_1\}$ and $\{d_0,d_1\}$ cross and by PTO, the arc $\{d_0,c_1\}$ lies in $\cX$. 
Setting $b_0 = d_0$ yields a point $b_0 \in [y_0,x_i]$, with $\{b_0,c_1\} \in \cX$.

Case 3: Assume $d_1 \in [c_0,d_0)$. Note that since $\{d_1,d_0\} \in \cX$, by $\Sigma$-closure and PTO we also have $\{d_1^{(-k)}, d_0\} \in \cX$ for all $k \geq 0$. Picking $m >0$ such that $d_1^{(-m)} = c_0^{-}$, we get the arc $\{d_0,d_1^{(-m)}\} \in \cX$ reducing the situation to the previous case.

Case 4: Assume $d_1 \in (d_0,c_1)$. If $d_1 \in (d_0^+,x_i]$, then $d_0 \in [c_0,d_1)$ and exchanging the roles of $d_0$ and $d_1$ in Case 3 we find $b_0 \in [y_0,x_i]$, with $\{b_0,c_1\} \in \cX$.
Else, if $d_1 \in (x_i,c_1)$ picking $m \geq 0$ such that $d_0^{(-m)} = c_0$, by $\Sigma$-closure of $\cX$ the arc $\{c_0,d_1^{-m}\}$ lies in $\cX$ and crosses $\{d_0,d_1\}$ and by PTO, the arc $\{c_0,d_1\}$ lies in $\cX$. The arc $\{c_0^-,d_1^-\}$ lies in $\cX$ as well and crosses $\{d_0,d_1\}$. By PTO we get that $\{c_0^-,d_0\}\in \cX$, hence we can reduce to Case 2 again.

In all the cases, we have found $b_0 \in [y_0,x_i]$ such that $\{b_0,c_1\} \in \cX$. Repeating this process with $b_0$ in the role of $c_1$ and $c_1$ in the role of $c_0$ yields the desired arc $\{b_0,b_1\} \in \cX$ with $b_0 \in [y_0,x_i]$ and $b_1 \in [y_1,x_j]$.
There exist an $m,n \in \bZ_{\geq 0}$ such that $b_1^{(-m)} = y_1$ and $b_0^{(-n)} = y_0$. In case $m=n$, we immediately get that $\{y_0,y_1\} \in \cX$. In case $m\neq n$ the arcs $\{b_0^{(-m)},y_1\}$ and $\{y_0,b_1^{(-n)}\}$ cross and belong to $\cX$, by  PTO we obtain $\{y_0,y_1\} \in \cX$ as desired.
\end{proof}

\subsection{Coaisles and Kreweras complement}

We will construct the coaisle of a $t$-structure on $\cC(\cZ)$ from its aisle using the Kreweras complement, which provides the self-duality of the lattice of non-crossing partitions.  Let us recall the definition of the Kreweras complement of a non-crossing partition $\cP=\{B_m\mid m\in I\}$. For that consider the poset $[2n] \cong \{1, 1', 2, 2', \ldots, n,n'\}$ with $1<1'<2<2'<\dots<n<n'$. The partition $\cP$ gives a non-exhaustive non-crossing partition of $[2n]$, if we identify elements with the same labels. There exists a unique maximal non-crossing partition $\tilde{\cP}=\{\tilde{B}_m\mid m\in \tilde{I} \}$  of $[2n]$, which contains all the blocks  $\cP$. The blocks $\tilde{\cP}\setminus \cP$ gives a non-crossing partition of the poset $[n] \cong \{1', 2', \ldots, n'\}$ with the induced order. We denote this partition of $[n]$ by $\cP^c$. It is called the {\bf Kreweras complement of $\cP$}, and was introduced by Kreweras in \cite{K}. From the construction, we see that after applying the construction of the Kreweras complement twice we obtain $\cP$ rotated once in the clockwise direction. 

\begin{remark}\label{R:Kreweras}
    Note that if $i \neq j$ are elements in $[n]$ which lie in the same block of a non-crossing partition $\cP$, then they do not lie in the same block of its Kreweras complement $\cP^c$. Else, in the associated non-crossing partition $\tilde{\cP}$ of $[2n] \cong \{1,1', \ldots, n,n'\}$, the block $B$ containing $i$ and $j$, and the block $B'$ containing $i'$ and $j'$ would cross.
\end{remark}

Let us describe the coaisle of a t-structure in terms of the combinatorial model. Figure \ref{fig:coaisle} provides an illustration of the aisle and the coaisle for the example from Figure \ref{fig:aisle}. The dark grey regions on the right hand side describe the coaisle, and mark the Kreweras complement $\cP^c = \{\{1,2\}, \{3,6\}, \{4\}, \{5\}\}$ of the non-crossing partition $\cP = \{\{1,3\}, \{2\}, \{4,5,6\}\}$ marked by the light grey areas on the left hand side describing the aisle. The heart of the corresponding t-structure is given by the four indecomposable objects $\{x_1^{--},x_1\}, \{x_3^{--},x_3\}, \{x_4^{--},x_4\}$ and $\{x_6^{--},x_6\}$ (see Corollary \ref{C:heart} for details). 

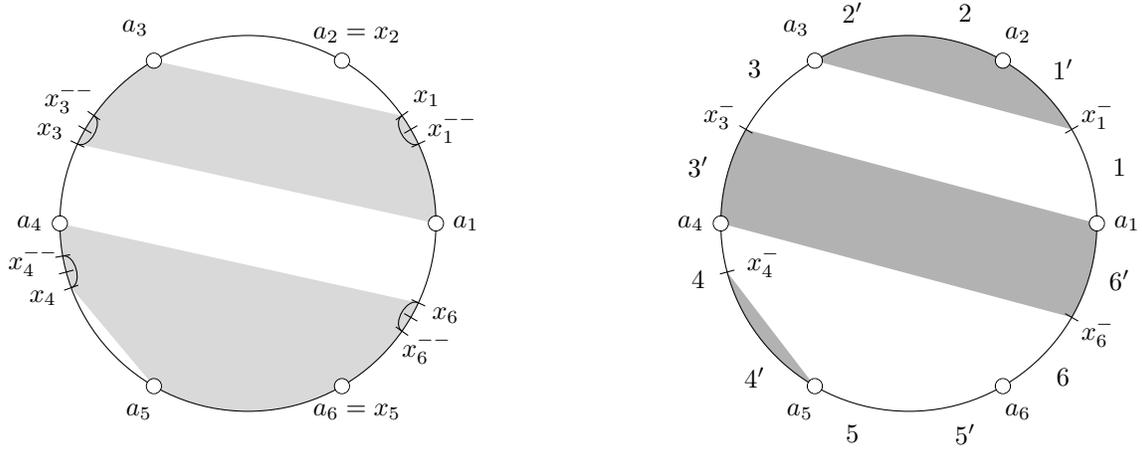
\begin{figure}[ht]
\begin{center}
\begin{tikzpicture}[scale=5]

	
	
	
	
	
	\fill[fill=gray!30]    (0:0.5) arc (0:35:0.5) -- (120:0.5) arc(120:155:0.5) -- (0:0.5);
	

\fill[fill=gray!30]    (180:0.5) arc (180:200:0.5) -- (240:0.5) arc(240:335:0.5) -- (180:0.5);

	
	
        \draw (0,0) circle(0.5cm);
        
	\draw (0:0.5) node[fill=white,circle,inner sep=0.065cm] {} circle (0.02cm);
	\node at (0:0.58){$a_1$};
	\draw (60:0.5) node[fill=white,circle,inner sep=0.065cm] {} circle (0.02cm);
	\node at (60:0.58){$a_2 = x_2$};
	\draw (120:0.5) node[fill=white,circle,inner sep=0.065cm] {} circle (0.02cm);
	\node at (120:0.6){$a_3$};
	\draw (180:0.5) node[fill=white,circle,inner sep=0.065cm] {} circle (0.02cm);
	\node[] (a_4) at (180:0.58){$a_4$};	
	\draw (240:0.5) node[fill=white,circle,inner sep=0.065cm] {} circle (0.02cm);
	\node[] (a_5) at (240:0.58){$a_5$};	
	\draw (300:0.5) node[fill=white,circle,inner sep=0.065cm] {} circle (0.02cm);
	\node[] (a_6) at (300:0.58){$a_6 = x_5$};	
	
	

	
	\draw (35:0.52) -- (35:0.48);
	\node (x_1) at (35:0.58){$x_{1}$};
	
	\draw (200:0.52)  -- (200:0.48);
	\node (x_4) at (200:0.58){$x_{4}$};
	
	\draw (155:0.52)  -- (155:0.48);
	\node (x_{3}) at (155:0.58){$x_3$};
	
	\draw (335:0.52)  -- (335:0.48);
	\node (x_{6}) at (335:0.58){$x_6$};
	
	
	\draw (30:0.52)  -- (30:0.48);
	\node (x_1-) at (30:0.58){};
	
	\draw (195:0.52)  -- (195:0.48);
	\node (x_4-) at (195:0.58){};
	
	\draw (150:0.52)  -- (150:0.48);
	\node (x_{3}-) at (150:0.58){};
	
	\draw (330:0.52)  -- (330:0.48);
	\node (x_{6}-) at (330:0.58){};
	
	
	\draw (25:0.52)  -- (25:0.48);
	\node (x_1--) at (25:0.6){$x_{1}^{--}$};
	
	\draw (190:0.52)  -- (190:0.48);
	\node (x_4--) at (190:0.58){$x_{4}^{--}$};
	
	\draw (145:0.52)  -- (145:0.48);
	\node (x_{3}--) at (145:0.58){$x_3^{--}$};
	
	\draw (325:0.52)  -- (325:0.48);
	\node (x_{6}--) at (325:0.58){$x_6^{--}$};
	
	
	\draw (25:0.5) to[out=180+25,in=180+35] (35:0.5);
	\draw (145:0.5) to[out=180+145,in=180+155] (155:0.5);
	\draw (190:0.5) to[out=180+190,in=180+200] (200:0.5);
	\draw (325:0.5) to[out=180+325,in=180+335] (335:0.5);
	
	\begin{scope}[xshift=50]
	
	\fill[fill=gray!60]    (30:0.5) arc (30:120:0.5) -- (30:0.5);
	
	\fill[fill=gray!60]    (150:0.5) arc (150:180:0.5) -- (330:0.5) arc(330:360:0.5) -- (150:0.5);
	
	\fill[fill=gray!60]    (195:0.5) arc (195:240:0.5) -- (195:0.5);

	
	
	

	
	
        \draw (0,0) circle(0.5cm);
        
	\draw (0:0.5) node[fill=white,circle,inner sep=0.065cm] {} circle (0.02cm);
	\node at (0:0.58){$a_1$};
	\draw (60:0.5) node[fill=white,circle,inner sep=0.065cm] {} circle (0.02cm);
	\node at (60:0.58){$a_2$};
	\draw (120:0.5) node[fill=white,circle,inner sep=0.065cm] {} circle (0.02cm);
	\node at (120:0.6){$a_3$};
	\draw (180:0.5) node[fill=white,circle,inner sep=0.065cm] {} circle (0.02cm);
	\node[] (a_4) at (180:0.58){$a_4$};	
	\draw (240:0.5) node[fill=white,circle,inner sep=0.065cm] {} circle (0.02cm);
	\node[] (a_5) at (240:0.58){$a_5$};	
	\draw (300:0.5) node[fill=white,circle,inner sep=0.065cm] {} circle (0.02cm);
	\node[] (a_6) at (300:0.58){$a_6$};	
	
	

	
	
	
	
	
	
	\draw (30:0.52)  -- (30:0.48);
	\node (x_1-) at (30:0.58){$x_1^-$};
	
	\draw (195:0.52)  -- (195:0.48);
	\node (x_4-) at (195:0.4){$x_4^-$};
	
	\draw (150:0.52)  -- (150:0.48);
	\node (x_{3}-) at (150:0.58){$x_{3}^-$};
	
	\draw (330:0.52)  -- (330:0.48);
	\node (x_{6}-) at (330:0.58){$x_{6}^-$};
	
	
	
	
	
	
	

    \node (1) at (15:0.58){$1$};
    \node (1') at (45:0.58){$1'$};
    \node (2) at (75:0.58){$2$};
    \node (2') at (105:0.58){$2'$};
    \node (3) at (135:0.58){$3$};
    \node (3') at (165:0.58){$3'$};
    \node (4) at (195:0.58){$4$};
    \node (4') at (225:0.58){$4'$};
    \node (5) at (255:0.58){$5$};
    \node (5') at (285:0.58){$5'$};
    \node (6) at (315:0.58){$6$};
    \node (6') at (345:0.58){$6'$};

	\end{scope}

  \end{tikzpicture}
  \end{center}
  \caption{
  An illustration of the aisle on the left, along with its coaisle on the right.}
\label{fig:coaisle}
\end{figure}

\begin{corollary}\label{C:coaisle}
    Let $(\cX,\cY)$ be a t-structure in $\cC(\cZ)$, with aisle $\cX$ corresponding to the $\overline{\cZ}$-decorated non-crossing partition $(\cP, \bf{x})$ with decoration ${\bf x} = (x_1, \ldots, x_n)$. 
    For $i = 1, \ldots, n$, set
    \[
        y_i = \begin{cases}
                x_i^{-} & \text{if } x_i \in \cZ\\
                x_i & \text{else}.
            \end{cases}
    \]
    and let $\cP^c = \{B'_m \mid m \in I'\}$ be the Kreweras complement of $\cP$. 
    Then
    \[
        \cY = \add \{\{z_0,z_1\} \; \text{arc of}\; \cZ \mid z_0,z_1 \in \bigcup_{i \in B'_m} [y_i,a_{i+1}) \; \text{for some} \; m \in I'\}.
    \]
\end{corollary}

\begin{proof}
Let us compute $\Sigma^{-1}\cY$. We show that 
\[
 \Sigma^{-1} \cY = \add\{\{z_0,z_1\} \; \text{arc of}\; \cZ \mid z_0,z_1 \in \bigcup_{i \in B'_m} [x_i,a_{i+1}) \; \text{for some} \; m \in I'\}.
\]
An arc $\{w_0,w_1\}$ belongs to $\Sigma^{-1}\cY$ if and only if it does not cross any arc from $\cX$. Assume that that is the case. Clearly $w_0$ and $w_1$ cannot lie in $(a_i,x_i)$ for any $i$, that is $w_0\in [x_i,a_{i+1})$ and $w_1\in [x_j,a_{j+1})$ for some $i$ and $j$. Assume $i$ and $j$ are not in the same block of $\cP^c$. Consider the non-crossing partition $\cP \cup \cP^c$ of $[2n] \cong \{1,1',2,2', \ldots, n,n'\}$, where $\cP$ is viewed as a non-crossing partition of $\{1, 2, \ldots, n\}$ and $\cP^c$ is viewed as a non-crossing partition of $\{1',2', \ldots, n'\}$. Thus $i'$ and $j'$ are not in the same block of $\cP \cup \cP^c$. Therefore, there must be $k$ and $l$ in the same block of $\cP$ such that $i' < k < j' < l$ or $k < i' < l < j'$. It follows that $x_i < x_k \leq x_j < x_l$ or $x_k \leq x_i < x_l \leq x_j$. 
Since the block of $\cP$ containing $k$ and $l$ is not a singleton, we can always find an arc in $\cX$ with endpoints in $(a_k,x_k]$ and $(a_l,x_l]$ which crosses $\{w_0,w_1\}$. This shows that the left hand side is contained in the right hand side.

Conversely, assume that $\{w_0,w_1\}$ lies in $\{\{z_0,z_1\} \; \text{arc of}\; \cZ \mid z_0,z_1 \in \bigcup_{i \in B'_m} [x_i,a_{i+1}) \; \text{for some} \; m \in I'\}$. If $i$ and $j$ are in the same block of $\cP^c$, we cannot find $k$ and $l$ with the properties as above and so the arc $\{w_0,w_1\}$ does not cross any arc from $\cX$ and thus lies in $\Sigma^{-1}\cY$. The claim follows.     
\end{proof}

The {\bf heart} of the t-structure $(\cX,\cY)$ is the subcategory $\cX \cap \Sigma \cY$. It is an abelian category, with the short exact sequences induced by the triangles in $\cC(\cZ)$ \cite{BBD}.

\begin{corollary}\label{C:heart}
    Let $(\cX,\cY)$ be a t-structure in $\cC(\cZ)$ with associated $\overline{\cZ}$-decorated non-crossing partition $(\cP,{\bf x})$ with $\bf x = (x_1, \ldots,  x_n)$. Its heart is the subcategory
    \[
        \add\{ \{x_i^{(-2)},x_i\} \mid x_i \in \cZ)\} \cong \mathrm{mod}(\underbrace{\bK\times \dots \times \bK}_{m \text{ times}}). 
    \]
    where $m = |\{i \in [n] \mid x_i \in \cZ\}|$. 
\end{corollary}

\subsection{Approximation triangles}
In this section we will describe approximation triangles for any t-structure on $\cC(\cZ)$. For that we will use a more detailed version of Lemma \ref{ConeInd} in a special situation.

\begin{lemma}\label{ParallelCone}
    Let $X=\{x,x'\}\in \cC(\cZ)$ and $Y_j=\{y_j,y'_j\}\in \cC(\cZ)$, $j=1,\dots,m$ be arcs such that the arcs $Y_j$ are mutually non-crossing, but they all cross the arc $X$. We can assume that \[x<y'_1<y'_2<\dots y'_m<x'<y_m\dots<y_2<y_1<x.\]
    Let $f: \bigoplus\limits_{j=1}^m Y_j   \rightarrow \Sigma X$ be a map such that all its components $f_j:  Y_j \rightarrow \Sigma X$ are not equal to zero. There exists a triangle
    \[
        X \rightarrow C \rightarrow \bigoplus_{j=1}^m Y_j \xrightarrow{f} \Sigma X,  
    \]
    such that $C$ is isomorphic to a direct sum of arcs $\{ x,y_1\}, \{ y_1',y_{2}\}, \{ y_2',y_{3}\},\dots, \{ y'_{m-1},y_{m}\}, \{ y_m',x'\}.$
\end{lemma}

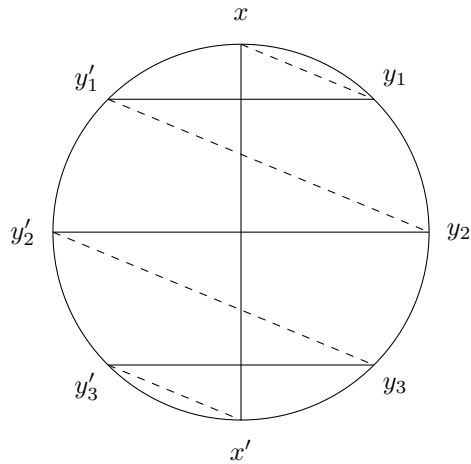
\begin{figure}[ht]
\begin{center}
\begin{tikzpicture}[scale=5]
        \draw (0,0) circle(0.5cm);
        
 	\node at (90:0.58){$x$};
 	\node at (135:0.58){$y'_1$};
 	\node at (180:0.58){$y'_2$};
 	\node at (225:0.58){$y'_3$};
 	\node at (270:0.58){$x'$};
 	\node at (315:0.58){$y_3$};
 	\node at (360:0.58){$y_2$};
 	\node at (45:0.58){$y_1$};
 	
 	    \draw (45:0.5) -- (135:0.5);
 		\draw (225:0.5) -- (315:0.5);
 		\draw (0:0.5) -- (180:0.5);
 		\draw (90:0.5) -- (270:0.5);
 		
 		\draw[dashed] (90:0.5) -- (45:0.5);
 		\draw[dashed] (135:0.5) -- (360:0.5);
 		\draw[dashed] (180:0.5) -- (315:0.5);
 		\draw[dashed] (225:0.5) -- (270:0.5);

  \end{tikzpicture}
  \end{center}
  \caption{An illustration for Lemma \ref{ParallelCone} in case $m=3$. The dashed arcs correspond to the object $C$.
\label{fig:zig-zag}}
\end{figure}

\begin{proof}
    We will prove the lemma by induction on $m$ again. The case $m=1$ follows from triangle \eqref{ExtOfArcs}. 
Now let us consider the following octahedral diagram in the general case.     
$$ 
 \xymatrix{
&\Sigma^{-1}\bigoplus\limits_{j=2}^m   Y_j \ar@{=}[r] \ar[d]^{\Sigma^{-1} g} &\Sigma^{-1}\bigoplus\limits_{j=2}^m  Y_j \ar[d] &\\
X\ar[r]^{w} \ar@{=}[d] & X_1 \ar[r] \ar[d]&Y_1 \ar[r]^{f_1} \ar[d]^{\iota_1}& \Sigma X  \ar@{=}[d]\\
X \ar[r]&C \ar[r] \ar[d]&\bigoplus\limits_{j=1}^mY_j \ar[r]^f \ar[d]^{\pi}& \Sigma X, \\
& \bigoplus\limits_{j=2}^m Y_j \ar@{=}[r]&\bigoplus\limits_{j=2}^mY_j&
}
$$    
where $\iota_1$ is the canonical inclusion of $Y_1$ into $\bigoplus\limits_{j=1}^mY_j$. The object $X_1$ decomposes into a direct sum of $\{x,y_1\}$ and $\{y_1',x'\}$ (note that the arc $\{x,y_1\}$ might be trivial). The arc $\{x,y_1\}$ does not cross any of the arcs $Y_2,\dots,Y_m$. By the octahedral axiom $w\Sigma^{-1}f=\Sigma^{-1}(g \pi) $. If some component $g_j: Y_j\rightarrow \Sigma \{y_1',x'\}$ of $g$ is zero, then the corresponding component $w\Sigma^{-1}f\mid_{Y_j}$ is zero. The component $w_1: X\rightarrow \{y_1',x'\}$ of $w$ is non-zero, thus $w_1$ and $\Sigma^{-1}f_j$, $j=2,\dots, m$ compose in a non-zero way by the description of the composition of morphisms in $\cC(\cZ)$. Hence all components $g_j$, $j=2,\dots, m$ are non-zero and we can apply the induction hypothesis to the triangle $$\{y_1',x'\} \rightarrow C' \rightarrow \bigoplus_{j=2}^mY_j \xrightarrow{g'} \Sigma \{y_1',x'\},$$ where $g'$ is the composition of $g$ and the projection from $\Sigma X_1$ to $\Sigma \{y_1',x'\}$. We get that $$C'\simeq \{ y_1',y_{2}\}\oplus \{ y_2',y_{3}\}\oplus \dots\oplus \{ y_{m-1}',y_{m}\}\oplus \{ y_m',x'\},$$ and hence $C\simeq \{x,y_1\} \oplus \{ y_1',y_{2}\}\oplus \{ y_2',y_{3}\}\oplus \dots\oplus \{ y_{m-1}',y_{m}\}\oplus \{ y_m',x'\}, \text{ as desired.}$  
\end{proof}

Let $(\cP,{\bf x})$ be a $\overline{\cZ}$-decorated non-crossing partition with $\cP(\cX) = \{B_m \mid m \in I\}$ and ${\bf x} (\cX)= (x_1, \ldots, x_n)$. For a marked point $u \in \cZ$ and a block $B \in \cP$ we set $u_B = \max_{[u^+,u]}\{z \in \bigcup_{i \in B}(a_i,x_i]\}$. This is the first point lying in one of the intervals $(a_i,x_i]$ indexed by an element in $B$ when walking from $u$ in a {\em clockwise} direction. Note that we have $u_B = u$ if and only if $u \in \bigcup_{i \in B}(a_i,x_i]$, and otherwise $u_B = \max_{[u^+,u]}\{x_i \mid i \in B\}$. Let now $T=\{t,t'\}$ be an arc in $\cC(\cZ)$. We say that $T$ and $B \in \cP$ cross, if there exists an arc $\{v_0,v_1\}$ with $v_0,v_1 \in \bigcup_{i \in B}(a_i,x_i]$ crossing $T$. Let $B_1, \ldots, B_l \in \cP$ be the list of mutually distinct blocks that cross $T$, and for $i = 1, \ldots, l$ set $z_i = t_{B_i}$ and $z'_i = t'_{B_i}$. Without loss of generality we can assume that $t \geq z_1 > z_2 > \ldots > z_l > t'$. Since any two arcs $\{z_i,z'_i\}$ and $\{z_j,z'_j\}$ do not cross by construction, this is equivalent to $t' \geq z'_l > \ldots > z'_2 > z'_1 > t$. We set 
\[
T|_{\cX} = \{\text{non-trivial arcs } \{z'_i,z_i\} \text{ for } 1 \leq i \leq l\}
\]
and 
\[
    T|_{\cY} = \{\text{non-trivial arcs }  \{t,z_1^-\}, \{{z'_l}^{-},t'\}, \{{z'}^{-}_i,z^{-}_{i+1}\} \text{ for } 1 \leq i \leq l-1 \}.
\]

The following corollary describes the approximation triangle with respect to $(\cX(\cP,{\bf x}),\cY(\cP,{\bf x}))$ for an indecomposable object of $\cC(\cZ)$. For an arbitrary object one can obtain the approximation triangle as a direct sum of approximation triangles for its summands.

\begin{corollary}\label{C:approx}
    Let $(\cX,\cY)$ be a t-structure on $\cC(\cZ)$ with the associated $\overline{\cZ}$-decorated non-crossing partition $(\cP,{\bf x})$. Let $T=\{t,t'\}$ be an arc in $\cC(\cZ)$, then the approximation triangle of $T$ with respect to $(\cX,\cY)$ is of the form
    \[
    Z\rightarrow T\rightarrow W \rightarrow \Sigma Z,
    \]
    where $Z$ is a direct sum of all arcs from $T|_{\cX}$ 
 and $W$ is a direct sum of all arcs from $T|_{\cY}$ constructed above. 
\end{corollary}

\begin{proof}
We have $t \geq z_1 > z_2 > \ldots > z_l > t' \geq z'_l > \ldots > z'_2 > z'_1 > t$. Hence every arc in $\Sigma (T|_{\cX})$ crosses $T$. Applying Lemma \ref{ParallelCone} to the map with non-zero components $\Sigma Z \rightarrow \Sigma T$ induced by the intersections of $T$ and the arcs $\{{z'}_i^-,{z}_i^-\}$ in $\Sigma (T|_{\cX})$
we immediately see that
\[
    Z\rightarrow T\rightarrow W \rightarrow \Sigma Z
    \]
is a triangle in $\cC(\cZ)$, and by construction $Z \in \cX$. Let us check that $W\in \cY$. This is equivalent to showing that for every arc $w \in T|_{\cY}$, the arc $\Sigma^{-1} w$ does not cross any arc in $\cX$. We first consider the case $w = \{t,z_1^-\}$.
If $t$ is in $\bigcup_{i \in B_m} [x_i,a_i)$ for some $B_m$, then $m =1$ and the arc $\{t,z_1^{-}\}$ is a trivial arc, and hence so is the arc $\Sigma^{-1}w$. Otherwise, the arc $\Sigma^{-1}w$, which is given by $\{t^+,z_1\}$ does not cross any of the arcs in $\cX$ by construction. The fact that $w = \{{z'}_l^-,t'\} \in \cY$ follows analogously. 

Let now $m \in \{1, \ldots, l-1\}$ and consider the arc $w = \{{z'}^-_m,z^-_{m+1}\}$. Assume for a contradiction that $\Sigma^{-1}w = \{z'_m,z_{m+1}\}$ crosses an arc $\{u_0,u_1\} \in \cX$. By construction, we have $z'_m = x_p$ and $z_{m+1} = x_q$ for some $p \in B_m$ and $q \in B_{m+1}$. Furthermore, we have $u_0 \in (a_k,x_k]$ and $u_1 \in (a_l,x_l]$ for some $k,l \in B \in \cP$. Exchanging the roles of $u_0$ and $u_1$ if necessary we have $u_0 \in (z_{m+1}, z'_m) = (z_{m+1},z_m] \cup (z_m,z'_m)$ and $u_1 \in (z'_m,z_{m+1})$ and we obtain the inequality
\[
    a_k < u_0 \leq x_k \leq x_p = z'_m < a_l < u_1 \leq x_l \leq x_q = z_{m+1}.
\] 
If we had $u_0 \in (z_m,z'_m)$, then $\{u_0,u_1\} \in \cX$ would cross $\{z_m,z'_m\} \in \cX$, forcing $B = B_m$. But then we must have $u_1 \in (z'_m,z'_{m+1})$, since otherwise $\{u_0,u_1\}$ and $\{z'_{m+1},z_{m+1}\}$ cross, contradicting the non-crossing of the blocks $B_m \neq B_{m+1}$. However, this in turn contradicts the maximality of $z'_m$ in the interval $[t,t']$. Therefore, we must have $u_0 \in (z_{m+1},z_m]$. Symmetrically, we show that we must have $u_1 \in (z'_m,z'_{m+1}]$.

In particular, this implies that the block $B$ crosses the arc $T = \{t,t'\}$, and so $B = B_i$ for some $1 \leq i \leq l$. However, if $1 \leq i \leq m$ then we would have $t \leq z'_1 < \ldots < z'_m < u_1 \leq t'$, contradicting the maximality of $z'_i$, and if $m+1 \leq i \leq l$, then we would have $t' \leq z_l < \ldots < z_{m+1} < u_0 \leq t$, contradicting the maximality of $z_i$.
\end{proof}

Figure \ref{fig:approximations} illustrates the construction from Corollary \ref{C:approx}.

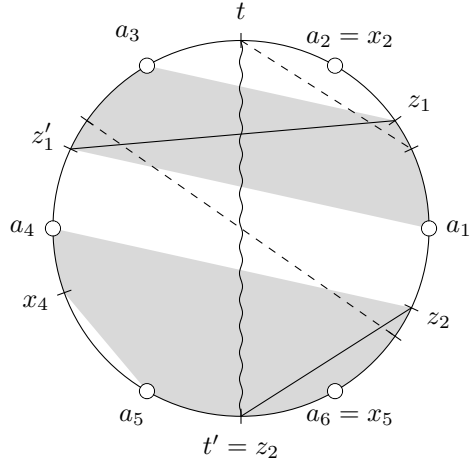
\begin{figure}[ht]
\begin{center}
\begin{tikzpicture}[scale=5]

	
	\fill[fill=gray!30]    (0:0.5) arc (0:35:0.5) -- (120:0.5) arc(120:155:0.5) -- (0:0.5);
	
	\fill[fill=gray!30]    (180:0.5) arc (180:200:0.5) to (240:0.5) arc(240:335:0.5) -- (180:0.5);

        \draw (0,0) circle(0.5cm);
        
	\draw (0:0.5) node[fill=white,circle,inner sep=0.065cm] {} circle (0.02cm);
	\node at (0:0.58){$a_1$};
	\draw (60:0.5) node[fill=white,circle,inner sep=0.065cm] {} circle (0.02cm);
	\node at (60:0.58){$a_2 = x_2$};
	\draw (120:0.5) node[fill=white,circle,inner sep=0.065cm] {} circle (0.02cm);
	\node at (120:0.6){$a_3$};
	\draw (180:0.5) node[fill=white,circle,inner sep=0.065cm] {} circle (0.02cm);
	\node[] (a_4) at (180:0.58){$a_4$};	
	\draw (240:0.5) node[fill=white,circle,inner sep=0.065cm] {} circle (0.02cm);
	\node[] (a_5) at (240:0.58){$a_5$};	
	\draw (300:0.5) node[fill=white,circle,inner sep=0.065cm] {} circle (0.02cm);
	\node[] (a_6) at (300:0.58){$a_6 = x_5$};

	
	\draw (35:0.52) -- (35:0.48);
	\node (x_1) at (35:0.58){$z_1$};
	
	\draw (200:0.52)  -- (200:0.48);
	\node (x_4) at (200:0.58){$x_4$};
	
	\draw (155:0.52)  -- (155:0.48);
	\node (x_{3}) at (155:0.58){$z'_1$};
	
	\draw (335:0.52)  -- (335:0.48);
	\node (x_{6}) at (335:0.58){$z_2$};
	
	
	\draw (35:0.5) to (155:0.5);
	\draw (335:0.5) to (270:0.5);
	\draw[dashed] (145:0.5) -- (325:0.5);
	\draw[dashed] (90:0.5) -- (25:0.5);

	
	\draw (25:0.52)  -- (25:0.48);
	\node (x_1-) at (25:0.58){};
	
	
	\draw (145:0.52)  -- (145:0.48);
	\node (x_{3}-) at (145:0.58){};
	
	\draw (325:0.52)  -- (325:0.48);
	\node (x_{6}-) at (325:0.58){};
	
	
	
	
	

    \draw (90:0.52)  -- (90:0.48);
	\node (t) at (90:0.58){$t$};
	
	\draw (270:0.52)  -- (270:0.48);
	\node (t') at (270:0.58){$t' = z_2$};

	\draw[decorate,decoration={snake,amplitude=.2mm}] (90:0.5) -- (270:0.5);

  \end{tikzpicture}
  \end{center}
  \caption{An illustration of Corollary \ref{C:approx} for the arc $T = \{t,t'\}$ and the aisle from Figure \ref{fig:aisle}. The straight arcs represent the indecomposable summands of $Z$ and the dashed arcs represent the indecomposable summands of $W$ in the approximation triangle $Z \to T \to W \to \Sigma Z$.
}
\label{fig:approximations}
\end{figure}

\begin{remark}
    We could have computed the approximation triangles for t-structure associated to $(\cP,{\bf x})$, using the factorization properties of morphisms in $\cC(\cZ)$, but we find Lemma \ref{ParallelCone} aesthetically pleasing. 
\end{remark}

\subsection{Non-degenerate and bounded t-structures}

Let $(\cX,\cY)$ be a t-structure in $\cC(\cZ)$. It is called {\bf left non-degenerate} if
$\bigcap_{n \in \bZ} \Sigma^n \cX = 0$,
and it is called {\bf right non-degenerate} if
$ \bigcap_{n \in \bZ} \Sigma^n \cY = 0$.
A t-structure that is both left and right non-degenerate is called {\bf non-degenerate}.

The description of non-degenerate t-structures in $\cC(\cZ)$ is an immediate consequence of Theorem \ref{T:one-to-one}, Corollary \ref{C:coaisle} and Corollary \ref{C:heart}.

\begin{corollary}\label{C:non-degenerate}
    Let $(\cX,\cY)$ be a t-structure  in $\cC(\cZ)$ associated to a $\overline{\cZ}$-decorated non-crossing partition $(\cP,{\bf x})$ with ${\bf x} = (x_1, \ldots, x_n)$. The t-structure $(\cX,\cY)$ is left non-degenerate if and only if either $x_i\in \cZ$ or $x_i=a_i$ for any $x_i\in{\bf x}$. It is right non-degenerate if and only if either $x_i\in \cZ$ or $x_i=a_{i+1}$ for any $x_i\in{\bf x}$.
    Finally, it is non-degenerate if and only if $x_i\in \cZ$ for any $x_i\in{\bf x}$, i.e.\
    if and only if the heart of $(\cX,\cY)$ is isomorphic to the module category of the product of $n$ copies of $\bK$.
\end{corollary}

Throughout the rest of this section, for completeness we explicitly include the case where $\cZ$ has exactly one limit point. As noted in Remark \ref{RemNg}, in this case the classification of t-structures only depends on the decoration ${\bf x}=(x_1)$ and the lattice of t-structures is isomorphic to $\mathbb{Z}\cup \{\pm\infty\}$ with linear order, which can be obtained by cutting $\overline{\cZ}$ at the limit point. 

We call a t-structure $(\cX,\cY)$ {\bf bounded above} if  $\bigcup_{n \in \bZ} \Sigma^n \cX = \cC(\cZ)$. We call it {\bf bounded below} if $\bigcup_{n \in \bZ} \Sigma^n \cY = \cC(\cZ)$. A t-structure is called {\bf bounded} if it is bounded above and bounded below.

\begin{remark}
    For $n=1$, every non-trivial t-structure $(\cX,\cY)$ (i.e.\ such that $\cX \neq 0$ and $\cY \neq 0$) in $\cC(\cZ)$ is bounded. 
\end{remark}

\begin{proposition}
    Let $n \geq 2$ and let $(\cX,\cY)$ be a t-structure in $\cC(\cZ)$ with associated $\overline{\cZ}$-decorated non-crossing partition $(\cP(\cX),\bf{x}(\cX)) = (\cP,\cX)$. Then it is 
    \begin{enumerate}
        \item bounded above if and only if $\cP$ is the coarsest partition of $[n]$: $\cP = \{[n]\}$;
        \item bounded below if and only if $\cP$ is the finest partition of $[n]$: $\cP = \{\{1\}, \{2\}, \ldots, \{n\}\}$.
    \end{enumerate}
\end{proposition}

\begin{proof}
    We start by showing (1). Assume first that $\cP \neq [n]$. Then we can find $i \neq j$ in $[n]$ which are in different blocks of $\cP = \{B_m \mid m \in I\}$. Pick $y_0 \in (a_i,a_{i+1})$ and $y_1 \in (a_j,a_{j+1})$ and consider $Y = \{y_0,y_1\} \in \cC(\cZ)$. For all $N \in \bZ$ we have
    \[
       Y \notin \Sigma^N \cX = \{\{z_0,z_1\} \; \text{arc of}\; \cZ \mid z_0,z_1 \in \bigcup_{i \in B_m} (a_i,x_i^{-N}] \; \text{for some} \; m \in I\},
    \]
    where by abuse of notation we set $a_k^{-N} = a_k$ for all $k \in [n]$. Thus $(\cX,\cY)$ is not bounded above.
    
    Conversely, assume $\cP = \{[n]\}$, and let $Y \in \cC(\cZ)$. Without loss of generality, since $\cC(\cZ)$ is Krull-Schmidt and both $\cX$ and $\cY$ are closed under direct summands, we can assume $Y$ to be indecomposable, that is $Y = \{y_0,y_1\}$ for some $y_0,y_1 \in \cZ$. Since $\cP = [n]$, there are no singletons in this partition, so $x_i \neq a_i$ for all $i \in [n]$. It follows that for $N \in \bZ_{\geq 0}$ big enough, we have $y_0 \in (a_i,x_i^N)$ and $y_1 \in (a_j,x_j^N)$ for some $i,j \in [n]$ (where, by abuse of notation, we set $a_k^N = a_k$ for $k \in [n]$, to cover the case where $x_i = a_{i+1}$ or $x_j = a_{j+1}$). In particular, $Y \in \Sigma^{-N} \cX$, and so $(\cX,\cY)$ is bounded above.
    
    The proof of (2) follows with a dual argument, using Corollary \ref{C:coaisle} and the observation that the finest and the coarsest partitions of $[n]$ are each other's mutual Kreweras complement.
 \end{proof}
 
 \begin{corollary}
    For $n \geq 2$ there exist no bounded t-structures in $\cC(\cZ)$.
\end{corollary}

\subsection{Equivalence classes of t-structures}

We use the notion of equivalence of t-structures introduced by Neeman. 

\begin{definition}[{\cite[Definition 0.10]{NeEqT}}]
	We say that two t-structures $(\cX,\cY)$ and $(\cX',\cY')$ are equivalent if there exists an $N \in \bZ_{>0}$ such that
	\[
		\Sigma^N \cX \subseteq \cX' \subseteq \Sigma^{-N} \cX.
	\]
\end{definition}

This motivates the following definitions.

\begin{definition}
    Let $(\cP,{\bf x})$ be a $\overline{\cZ}$-decorated non-crossing partition with ${\bf x} = (x_1, \ldots, x_l)$. We call the index set $Z({\bf x}) = \{i \in [n] \mid x_i \in \cZ\}$ the {\bf $\cZ$-indices of ${\bf x}$}. 
\end{definition}

\begin{definition}\label{D:equiv partitions}
	Let $(\cP, {\bf x})$ and $(\cP', {\bf x}')$ be $\overline{\cZ}$-decorated non-crossing partitions of $[n]$ with ${\bf x} = (x_1, \ldots, x_n)$ and ${\bf x}' = (x'_1, \ldots, x'_n)$. We say that $(\cP, {\bf x})$ and $(\cP', {\bf x}')$ are {\bf equivalent}, and write
	$(\cP, {\bf x}) \sim (\cP', {\bf x}')$
	if $\cP = \cP'$ and $Z({\bf x}) = Z({\bf x}')$.
\end{definition}

\begin{remark}\label{R:equal}
    Note that if $(\cP, {\bf x}) \sim (\cP', {\bf x}')$ and $i \notin Z({\bf x}) = Z({\bf x}')$, then automatically $x_i = x'_i$.
	Indeed, we have $x_i \notin \cZ$ if and only if $i$ is either a singleton or an adjacency in $\cP = \cP'$. In the former case, we must have $x'_i = a_i = x_i$ and in the latter case $x'_i = a_{i+1} = x_{i+1}$.
\end{remark}

We obtain the following characterisation of equivalent t-structures on $\cC(\cZ)$.

\begin{proposition}\label{P:equivalence t-structures}
	Let $(\cX,\cY)$ and $(\cX',\cY')$ be t-structures on $\cC(\cZ)$ with associated $\overline{\cZ}$-decorated non-crossing partitions $(\cP,{\bf x})$ and $(\cP',{\bf x}')$ respectively. Then $(\cX,\cY)$ and $(\cX',\cY')$ are equivalent if and only if $(\cP, {\bf x})$ and $(\cP', {\bf x}')$ are equivalent.
\end{proposition}

\begin{proof}
	 Consider $\cX$ and $\cX'$ as sets of arcs, and let $\cP = \{B_m \mid m \in I\}$ and $\cP' = \{B'_m \mid m \in I'\}$.
	We have
	\begin{eqnarray*}
		\cX = \{\{y_0,y_1\} \; \text{arc of}\; \cZ \mid y_0,y_1 \in \bigcup_{i \in B_m} (a_i,x_i] \; \text{for some} \; m \in I\} \\
		\cX' = \{\{y_0,y_1\} \; \text{arc of}\; \cZ \mid y_0,y_1 \in \bigcup_{i \in B'_m} (a_i,x'_i] \; \text{for some} \; m \in I'\},
	\end{eqnarray*}
	for ${\bf x} = (x_1,\ldots,x_n)$ and ${\bf x}' = (x'_1, \ldots, x'_n)$. For $k \in \bZ$ observe that the set of arcs associated to $\Sigma^k \cX$ is
	\begin{eqnarray}\label{E:shift of set of arcs}
		\cX^{-k} = \{\{y_0,y_1\} \; \text{arc of}\; \cZ \mid y_0,y_1 \in \bigcup_{i \in B_m} (a_i,x^{-k}_i] \; \text{for some} \; m \in I\}.
	\end{eqnarray}
	
	Assume first that $(\cP,{\bf x}) \sim (\cP',{\bf x}')$. For each $i \in [n]$ we set 
	\[
	    N_i = \begin{cases}
	                0 & \text{if $x_i=x'_i$}\\
	                m & \text{if $x_i \neq x'_i$ and $m \geq 0$ is such that $x_i^{(m)} = x'_i$ or ${x'}_i^{(m)} = x_i$.}
	            \end{cases}
	\]
	Note that, if $x_i \neq x'_i$, then necessarily $x_i,x'_i \in (a_i,a_{i+1}) \subseteq \cZ$ and $N_i$ is well-defined.
Setting $N = \max\{N_i \mid i \in [n]\}$ we observe that
	\begin{eqnarray*}
		a_i \leq x_i{^{-N}} \leq x_i^{-N_i} \leq x_i' \leq a_{i+1} & \text{and} &
		a_i \leq x_i' \leq x_i^{N_i} \leq x_i^{N} \leq a_{i+1}.
	\end{eqnarray*}
	In particular, for all $i \in [n]$ we have $(a_i,x_i^{-N}] \subseteq (a_i,x'_i] \subseteq (a_i,x_i^N]$. Comparing with (\ref{E:shift of set of arcs}) for $k = \pm N$ it follows that $\Sigma^N \cX \subseteq \cX' \subseteq \Sigma^{-N} \cX$. Therefore, $(\cX,\cY)$ and $(\cX',\cY')$ are equivalent.
	
	On the other hand, assume now that $(\cX,\cY)$ and $(\cX',\cY')$ are equivalent t-structures with $\cX^{-N} \subseteq \cX' \subseteq \cX^{N}$ for the associated sets of arcs. We first show that $\cP = \cP'$. Let $i \neq j \in [n]$ be in the same block of $\cP$. Then there exists an arc $\{y_0,y_1\} \in \cX$ with $y_0 \in (a_i,x_i]$ and $y_1 \in (a_j,x_j]$. The arc $\{y_0^{-N},y_1^{-N}\}$ lies in the set of arcs $\cX^{-N} \subseteq \cX'$. Thus, we must have $y_0^{-N} \in (a_i,x'_i]$ and $y_1^{-N} \in (a_j,x'_j]$, and $i$ and $j$ are in the same block of $\cP'$. A symmetric argument, using the inclusion $\cX' \subseteq \cX^N$, shows that if $i$ and $j$ are in the same block of $\cP'$, then they are in the same block of $\cP$. Therefore we have $\cP = \cP'$.
	
Finally, assume that for some $1 \leq i \leq n$ we have $x'_i \notin \cZ$, that is $x'_i = a_i$ or $x'_i = a_{i+1}$. In that case, for all $k \in \bZ$ we have $(a_i,{x'}_i^{k}] = (a_i,x'_i]$.  Using the inclusion $\cX^{-N} \subseteq \cX' \subseteq \cX^{N}$ and Equation (\ref{E:shift of set of arcs}) we must also have $(a_i,x_i] = (a_i,x_i^k]$ for all $k \in \bZ$ and hence $x_i \notin \cZ$. 
Therefore, we have $Z({\bf x}) = Z({\bf x}')$.
\end{proof}

Note that for a $\overline{\cZ}$-decoration ${\bf x}$ of $\cP$ the ${\cZ}$-indices $Z({\bf x})$ must contain all the elements $i \in [n]$ that are neither singletons nor adjacencies of $\cP$. 
Equivalence classes of $\overline{\cZ}$-decorated non-crossing partitions of $[n]$ are given by pairs $(\cP,Z)$, where $\cP$ is a non-crossing partition of $[n]$ and $Z \subseteq [n]$ contains all elements of $[n]$ that are neither singletons nor adjacencies of $\cP$.
A $\overline{\cZ}$-decorated non-crossing partition $(\cP,{\bf x})$ of $[n]$ is a representative of the equivalence class $(\cP,Z({\bf x}))$.

Proposition \ref{P:equivalence t-structures} yields the following classification of equivalence classes of t-structures of $\cC(\cZ)$.

\begin{corollary}
	The one-to-one correspondence between t-structures on $\cC(\cZ)$ and $\overline{\cZ}$-decorated non-crossing partitions of $[n]$ induces a one-to-one correspondence between equivalence classes of  t-structures of $\cC(\cZ)$ and pairs $(\cP,Z)$ such that $\cP \in NC_n$ and $Z$ is a subset of $[n]$ that contains all elements of $[n]$ that are neither singletons nor adjacencies of $\cP$. 
\end{corollary}

We immediately obtain the following corollaries.

\begin{corollary}\label{C:non-degenerate equivalence}
	The one-to-one correspondence between t-structures of $\cC(\cZ)$ and $\overline{\cZ}$-decorated non-crossing partitions of $[n]$ induces a one-to-one correspondence between equivalence classes of non-degenerate t-structures of $\cC(\cZ)$ and non-crossing partitions of $[n]$.
\end{corollary}

\begin{corollary}
For $n\geq 2$ all non-degenerate bounded below t-structures in $\cC(\cZ)$ are equivalent. Similarly, all non-degenerate bounded above t-structures in $\cC(\cZ)$ are equivalent.
\end{corollary}

\section{Lattice structure}

We show that the set of t-structures of $\cC(\cZ)$ forms a lattice under inclusion of aisles, with the meet of two t-structures given by the intersection of their aisles. Note that neither of these facts is generally true in an arbitrary triangulated category.

Kreweras \cite{K} showed that the set of non-crossing partitions of $[n]$ forms a lattice under refinement. Our construction of the lattice of non-exhaustive non-crossing partition is a slight generalisation of Kreweras' construction:
The order and the meet operation are defined in the same way,  cf Section \ref{S:thicksubcats}.

We know from Theorem \ref{T:one-to-one} that every t-structure is uniquely determined by a pair $(\cP,{\bf x})$, where $\cP$ is a non-crossing partition of $[n]$ and ${\bf x}$ is a $\overline{\cZ}$-decoration of $\cP$. The following lemma follows immediately from the correspondence between t-structures and $\overline{\cZ}$-decorated non-crossing partitions.  

\begin{lemma}\label{L:order}
	Let $\cX$ and $\cX'$ be aisles of t-structures in $\cC(\cZ)$ with associated $\overline{\cZ}$-decorated partitions $(\cP,{\bf x})$ and $(\cP',{\bf x}')$ respectively. Then $\cX \subseteq \cX'$ if and only if
	\[
		(\cP,{\bf x}) \leq (\cP',{\bf x}'),
	\]
	 i.e.\ if and only if $\cP \leq \cP'$ and for ${\bf x} = (x_1, \ldots, x_n)$, ${\bf x}' = (x'_1, \ldots, x'_n)$ we have $a_i \leq x_i \leq x'_i \leq a_{i+1}$ for all $1 \leq i \leq n$.
\end{lemma}

For ease of notation, given two arbitrary non-crossing partitions $\cP$ and $\cP'$ with $\overline{\cZ}$-decorations ${\bf x} = (x_1, \ldots, x_n)$ and ${\bf x'} = (x'_1, \ldots, x'_n)$, we write
\begin{align*}
	\min\{{\bf x}, {\bf x'}\} = (y_1, \ldots, y_n), \; \text{where} \; y_i = \min_{[a_i,a_{i+1}]}\{x_i,x'_i\}, \\
	\max\{{\bf x}, {\bf x'}\} = (z_1, \ldots, z_n), \; \text{where} \; z_i = \max_{[a_i,a_{i+1}]}\{x_i,x'_i\}.
\end{align*}

\begin{thm}\label{T:lattice}
	The set of t-structures of $\cC(\cZ)$ forms a lattice under inclusion of aisles. 
	More precisely, consider two t-structures $(\cX, \cY)$ and $(\cX',\cY')$, with associated $\overline{\cZ}$-decorated non-crossing partitions $(\cP,{\bf x})$ and $(\cP',{\bf x}')$ respectively. Then their meet is given by the t-structure with the aisle 
	\[
		\cX(\cP \wedge \cP', \min\{{\bf x}, {\bf x'}\}),
	\] 
	and their join is given by the t-structure with the aisle 
	\[
		\cX(\cP \vee \cP', \max\{{\bf x}, {\bf x'}\}).
	\] 
\end{thm}

\begin{proof}
	Using our classification from Theorem \ref{T:one-to-one},  both $\cX(\cP \wedge \cP', \min\{{\bf x}, {\bf x'}\})$ and $\cX(\cP \vee \cP', \max\{{\bf x}, {\bf x'}\})$ are clearly aisles of t-structures, since $\min\{{\bf x}, {\bf x'}\}$ is a $\overline{\cZ}$-decoration of $\cP \wedge \cP'$ and $\max\{{\bf x}, {\bf x'}\}$ is a $\overline{\cZ}$-decoration of $\cP \vee \cP'$. 
    The fact that the t-structure with the aisle $\cX \vee \cX' = \cX(\cP \vee \cP', \max\{{\bf x}, {\bf x'}\})$ is the join of $(\cX, \cY)$ and $(\cX',\cY')$ and that the t-structure with the aisle $\cX(\cP \wedge \cP', \min\{{\bf x}, {\bf x'}\})$ is the meet of $(\cX, \cY)$ and $(\cX',\cY')$ follows directly from Lemma \ref{L:order}. 
\end{proof}

The meet, as described in Theorem \ref{T:lattice}, is in fact given by the intersection of the aisles.

\begin{proposition}\label{P:meet is intersection}
	With the notation from Theorem \ref{T:lattice} we have $\cX(\cP \wedge \cP', \min\{{\bf x}, {\bf x'}\}) = \cX \cap \cX'$. 
\end{proposition}

\begin{proof}
	Let $\cP = \{B_m \mid m \in I\}$ and $\cP' = \{B'_m \mid m \in I'\}$ and ${\bf x} = (x_1, \ldots, x_n)$, ${\bf x}' = (x_1', \ldots, x_n')$. Then the subcategory $\cX \cap \cX'$ corresponds to the set of arcs $\cX \cap \cX'$, where
	\begin{align*}
		\cX = \{\{y_0,y_1\} \; \text{arc of}\; \cZ \mid y_0,y_1 \in \bigcup_{i \in B_m} (a_i,x_i] \; \text{for some} \; m \in I\} \text{ and} \\
		\cX' = \{\{y_0,y_1\} \; \text{arc of}\; \cZ \mid y_0,y_1 \in \bigcup_{i \in B'_m} (a_i,x'_i] \; \text{for some} \; m \in I'\}.
	\end{align*}
We have $\cP \wedge \cP' = \{B_m \cap B'_{m'} \mid m \in I, m' \in I', B_m \cap B'_{m'} \neq \varnothing\}$. It is thus enough to show that
	\begin{align}\label{E:equality}
		\cX \cap \cX' = \{\{y_0,y_1\} \; \text{arc of}\; \cZ \mid y_0,y_1 \in \bigcup_{i \in B_m \cap B'_{m'}} (a_i,\min_{[a_i,a_{i+1}]}\{x_i,x'_i\}] \; \text{for some} \; m \in I, m' \in I'\}.
	\end{align}
	We have $\{y_0,y_1\} \in \cX \cap \cX'$ if and only if $y_0, y_1 \in \big(\bigcup_{i \in B_m}(a_i,x_i]\big) \cap \big(\bigcup_{j \in B'_{m'}}(a_j,x'_j]\big)$ for some $m \in I$ and $m' \in I'$. This is the case if and only if $y_0 \in (a_i,x_i] \cap (a_i,x'_i]$ and $y_1 \in (a_j,x_j] \cap (a_j,x'_j]$ and $i,j \in B_m \cap B'_{m'}$ for some $m \in I$ and $m' \in I'$. Observing that $(a_i,x_i] \cap (a_i,x'_i] = (a_i,\min_{[a_i,a_{i+1}]}\{x_i,x'_i\}]$ and $(a_j,x_j] \cap (a_j,x'_j] = (a_j,\min_{[a_j,a_{j+1}]}\{x_j,x'_j\}]$ shows the equality in (\ref{E:equality}).
\end{proof}

\begin{remark}
    By \cite{K}, the lattice of non-crossing partitions is self-dual. In fact, so is the lattice of $\overline{\cZ}$-decorated non-crossing partitions. This corresponds to the fact that  $\cC(\cZ)$ is equivalent to $\cC(\cZ)^{op}$. This equivalence sends t-structures to t-structures and aisles to coaisles. So the lattice of t-structures under inclusion of aisles in $\cC(\cZ)$ is isomorphic to the lattice of t-structures under inclusion of coaisles in $\cC(\cZ)^{op}\simeq \cC(\cZ)$, forcing self-duality of the lattice of $\overline{\cZ}$-decorated t-structures. The join in the lattice of t-structures corresponds to the intersection of coaisles.
\end{remark}

Let now $\cT$ be any triangulated category. If t-structures form a lattice under inclusion of aisles, then this induces a lattice structure on the equivalence classes of t-structures. 

\begin{remark}
    The poset of t-structures on $\cT$ under inclusion of aisles induces a poset structure on the equivalence classes of t-structures: If $[(\cX,\cY)]$ and $[(\cX',\cY')]$ are equivalence classes of t-structures then we say $[(\cX,\cY)] \leq [(\cX',\cY')]$ if for every $(\tilde{\cX},\tilde{\cY}) \in [(\cX,\cY)]$ there exists a $(\tilde{\cX}',\tilde{\cY}') \in [(\cX',\cY')]$ such that $\tilde{\cX} \subseteq \tilde{\cX}'$.
\end{remark}

\begin{lemma}
    If the t-structures on $\cT$ form a lattice under inclusion of aisles, then the induced partial order on equivalence classes of t-structures on $\cT$ is also a lattice.
\end{lemma}

\begin{proof}
    We show that the meet of two equivalence classes is given by the equivalence class of the meet of their respective representatives, the proof for the joins is analogous. First let us check that these operations are well-defined. Assume we have t-structures $(\cX,\cY) \sim (\tilde{\cX},\tilde{\cY})$ with $\Sigma^N\tilde{\cX} \subseteq \cX \subseteq \Sigma^{-N}\tilde{\cX}$ and $(\cX',\cY') \sim (\tilde{\cX}',\tilde{\cY}')$ with $\Sigma^M\tilde{\cX}' \subseteq \cX' \subseteq \Sigma^{-M}\tilde{\cX}'$. Without loss of generality we can assume $M = N$. Then we get that $\Sigma^N(\tilde{\cX} \wedge \tilde{\cX}') \subseteq \Sigma^N \tilde{\cX} \subseteq \cX$ and $\Sigma^N(\tilde{\cX} \wedge \tilde{\cX}') \subseteq \Sigma^N \tilde{\cX}' \subseteq \cX'$, and thus $\Sigma^N(\tilde{\cX} \wedge \tilde{\cX}') \subseteq \cX' \wedge \cX$. Analogously, we obtain $\cX \wedge \cX' \subseteq \Sigma^{-N}(\tilde{\cX} \wedge \tilde{\cX}')$. Therefore, the t-structures with aisles $\cX \wedge \cX'$ and $\tilde{\cX} \wedge \tilde{\cX}'$ respectively are equivalent.

     Consider now an equivalence class $[(\cX'',\cY'')]$ of t-structures that is smaller or equal to both $[(\cX,\cY)]$ and $[(\cX',\cY')]$. Take a representative $(\tilde{\cX}'',\tilde{\cY}'')$ of $[(\cX'',\cY'')]$. Then there exist representatives $(\tilde{\cX},\tilde{\cY)}$ and $(\tilde{\cX}',\tilde{\cY}')$ of $[(\cX,\cY)]$ and $[(\cX',\cY')]$ respectively such that $\tilde{\cX}'' \subseteq \tilde{\cX}$ and $\tilde{\cX}'' \subseteq \tilde{\cX}'$. Hence $\tilde{\cX}'' \subseteq \tilde{\cX} \wedge \tilde{\cX}'$ and hence $[(\cX'',\cY'')]$ is smaller or equal  to the equivalence class of the t-structure with the aisle $\cX \wedge \cX'$. Therefore what we defined is indeed the meet.
\end{proof}

The following is an immediate consequence of Corollary \ref{C:non-degenerate equivalence} and Proposition \ref{P:equivalence t-structures}.

\begin{corollary}\label{C:lattice of equivalence classes}
    The lattice of equivalence classes of non-degenerate t-structures is isomorphic to the lattice of non-crossing partitions of $[n]$. Its top is given by the equivalence class of non-degenerate bounded below t-structures, and its bottom by the equivalence class of non-degenerate bounded above t-structures.
\end{corollary}

	\bibliography{bibliography}{}
	\bibliographystyle{alpha}

\end{document}